\documentclass[hidelinks,onefignum,onetabnum]{siamart220329}

\usepackage{lipsum}
\usepackage{amsfonts}
\usepackage{graphicx}
\usepackage{epstopdf}
\usepackage{algorithmic}
\usepackage{amsmath,amssymb,enumerate,enumitem,hyperref,cancel,xcolor,array,tabularx,bbm,esint,cleveref,stackengine,overpic,caption,rotating}
\usepackage[caption=false]{subfig} 
\usepackage{listings}
\usepackage[mathscr]{euscript}
\ifpdf
  \DeclareGraphicsExtensions{.eps,.pdf,.png,.jpg}
\else
  \DeclareGraphicsExtensions{.eps}
\fi

\newsiamremark{remark}{Remark}
\newsiamremark{problem}{Problem}
\newsiamremark{hypothesis}{Hypothesis}
\crefname{hypothesis}{Hypothesis}{Hypotheses}
\newsiamthm{claim}{Claim}

\newcommand{\R}{\mathbb{R}}
\newcommand{\C}{\mathbb{C}}

\newcommand{\norm}[1]{\left\lVert#1\right\rVert}

\newcommand{\comment}[1]{}

\headers{Leveraging the HNA and block-AAA algorithms in ROM}{A. Yu and A. Townsend}

\title{Leveraging the Hankel norm approximation and block-AAA algorithms in reduced order modeling\thanks{This work is supported by National
Science Foundation grants DMS-1952757 and DMS-2045646.}}

\author{Annan Yu\thanks{Center for Applied Mathematics, Cornell University, Ithaca, 14853, NY, USA (\email{ay262@cornell.edu}).}
\and Alex Townsend\thanks{Department of Mathematics, Cornell University, Ithaca, 14853, NY, USA 
  (\email{townsend@cornell.edu}).}}

\usepackage{amsopn}
\DeclareMathOperator{\diag}{diag}
\DeclareMathOperator{\Iner}{In}
\usepackage{tikz}
\usetikzlibrary{shapes.geometric, arrows,positioning}

\tikzstyle{io} = [trapezium, 
trapezium stretches=true, 
trapezium left angle=70, 
trapezium right angle=110, 
minimum width=1.2cm, 
minimum height=0.8cm, text centered, 
draw=black, fill=blue!30]

\tikzstyle{process} = [rectangle, 
minimum width=2cm, 
minimum height=1.2cm, 
text centered, 
text width=2cm, 
draw=black, 
fill=orange!30]

\tikzstyle{decision} = [diamond, 
minimum width=0.5cm, 
minimum height=0.5cm, 
text centered, 
draw=black, 
fill=green!30]
\tikzstyle{arrow} = [thick,->,>=stealth]

\ifpdf
\hypersetup{
  pdftitle={xxx},
  pdfauthor={A. Yu, A. Townsend}
}
\fi

\begin{document}

\maketitle

\begin{abstract}
Large-scale linear, time-invariant (LTI) dynamical systems are widely used to characterize complicated physical phenomena. We propose a two-stage algorithm to reduce the order of a large-scale LTI system given samples of its transfer function for a target degree $k$ of the reduced system. In the first stage, a modified adaptive Antoulas--Anderson (AAA) algorithm is used to construct a degree $d$ rational approximation of the transfer function that corresponds to an intermediate system, which can be numerically stably reduced in the second stage using ideas from the theory on Hankel norm approximation (HNA). We also study the numerical issues of Glover's HNA algorithm and provide a remedy for its numerical instabilities. A carefully computed rational approximation of degree $d$ gives us a numerically stable algorithm for reducing an LTI system, which is more efficient than SVD-based algorithms and more accurate than moment-matching algorithms.
\end{abstract}



\section{Introduction} Dynamical systems are successful tools in modeling the complicated physics of nature. Among a wide range of dynamical systems, the linear, time-invariant (LTI) dynamical systems comprise one of the most important classes because they are relatively approachable in terms of analysis and simulation. Yet, they are expressive enough to model many physical phenomena. The applications of LTI dynamical systems range from signal and image processing~\cite{diniz2010digital,inouye1998criteria} to electrical and mechanical engineering~\cite{siebert1986circuits,wood1995control}. Unfortunately, while many theories and algorithms exist for LTI systems, a system that accurately models a complex physical phenomenon usually requires many parameters. To reduce the number of parameters, one approximates the large LTI system by a relatively small LTI system while maintaining as much information as possible. Compressing an LTI system is called model reduction and has been studied for decades. The two major concerns in computing reduced-order models (ROMs) are accuracy and efficiency.

In this paper, we propose to combine the fast-but-less-accurate block-adaptive Antoulas--Anderson (block-AAA) algorithm~\cite{nakatsukasa2018AAA} with the accurate-but-less-efficient Hankel norm approximation (HNA) algorithm~\cite{glover1984optimal}. Fixing a degree of approximation, we do not know an error bound for the AAA algorithm. Hence, we keep running the block-AAA algorithm until it yields an accurate approximation. Then, we use the HNA algorithm, whose a priori error bound is known, to reduce the approximation to the desired degree. Our new algorithm outperforms the block-AAA algorithm in accuracy and beats the HNA algorithm in computational time. To combine the two algorithms in a numerically stable way, we need to regularize the block-AAA algorithm. Moreover, we deal with some numerical issues in the HNA algorithm to make the reduced system numerically robust.

More formally, a continuous-time LTI dynamical system is given by~\cite{antoulas2005approximation,zhou1998essentials}
\begin{equation}\label{eq.LTIDS}
	\begin{aligned}
		x'(t) &= Ax(t) + Bu(t), \\
		y(t) &= Cx(t) + Du(t),
	\end{aligned}
	\qquad t \geq 0,
\end{equation}
where $A \in \C^{n \times n}, B \in \C^{n \times m}, C \in \C^{p \times n}$, $D \in \C^{p \times m}$, and $n,m,p$ are fixed positive integers. Here, $u(t) \in \C^m, x(t) \in \C^n$, and $y(t) \in \C^p$ are the vectors of inputs, states, and outputs, respectively. The system is called stable if the spectrum of $A$ is contained in the open left half-plane, i.e., $\text{Re}(z) < 0$ for all $z \in \sigma(A)$. Associated with the LTI system $(A,B,C,D)$ is the so-called transfer function defined by
\begin{equation}\label{eq.transfer}
	G(z) = D + C(zI - A)^{-1}B,
\end{equation}
where we used $I$ for an identity matrix. The transfer function in~\cref{eq.transfer} maps the inputs of the system to the outputs in the Laplace domain by multiplication if $x(0) = 0$, i.e.,
\begin{equation}\label{eq.transferlaplace}
	(\mathcal{L}y)(z) = G(z) (\mathcal{L}u)(z),
\end{equation}
where $\mathcal{L}$ is the Laplace transform operator~\cite{zhou1998essentials}. The task of model reduction is to find a reduced LTI system $(\hat{A},\hat{B},\hat{C},\hat{D})$ with the transfer function $\hat{G}$, where $\hat{A} \in \C^{k \times k}$ for some $k \ll n$, that well-approximates the original system. The notion of a good approximation has a couple of different interpretations, among which the most widely used measurements are the $H^2$, $H^\infty$, and Hankel error of the transfer functions~\cite{debruyne1995relating,glover1984optimal}, all of which measure the difference between $G$ and $\hat{G}$ in a (semi)norm. This makes function approximation an important theme in most model reduction algorithms.

Given a set of samples $\{(z_i,G(z_i))\}_{i=1}^N$ of the transfer function corresponding to the system in~\cref{eq.LTIDS}, we propose an algorithm for finding a stable reduced-order model $(\hat{A},\hat{B},\hat{C},\hat{D})$ of degree $k \ll n$. The algorithm, summarized in~\Cref{fig:flowchart}, has two stages (see~\Cref{fig:flowchart}):
\begin{itemize}
\item \textbf{Stage I:} We use the block-AAA algorithm~\cite{nakatsukasa2018AAA} to find a degree-$d$ rational approximation $R_d$ of the transfer function $G$ and extract a stable intermediate LTI system $(\tilde{A},\tilde{B},\tilde{C},\tilde{D})$ corresponding to the transfer function $R_d$.
\item \textbf{Stage II:} Given the intermediate system $(\tilde{A},\tilde{B},\tilde{C},\tilde{D})$ from stage I, we apply the HNA algorithm~\cite{glover1984optimal} to find the final reduced system $(\hat{A},\hat{B},\hat{C},\hat{D})$.
\end{itemize}
Model reduction algorithms based on approximating the transfer function $G$ are called moment-matching algorithms. Typical examples are the Krylov subspace methods based on the Pad{\'e} approximation~\cite{grimme1997krylov,jaimoukha1997implicitly}, the Loewner framework~\cite{antoulas2017tutorial,peherstorfer2016data}, the Astolfi framework~\cite{astolfi2010model,ionescu2014families}, and the AAA algorithm~\cite{gosea2021algorithms,lietaert2022automatic,nakatsukasa2018AAA} that we use in our first stage.  On the other hand, the HNA algorithm that we use in our second stage, along with several different popular model reduction algorithms such as balanced truncation~\cite{dullerud2013course,moore1981principal} and singular perturbation approximation~\cite{kokotovic1976singular,liu1989singular}, are based on computing the SVD of a square matrix of size $n \times n$. Such methods are referred to as SVD-based algorithms.

In our two-stage algorithm, the degree $d$ is a parameter we can choose. We consider the two extreme cases to understand the effect of $d$ and the benefits of the two-stage algorithm. For ease of exposition, we assume the system is single-in, single-out (SISO), i.e., $m = p = 1$. On the one hand, by setting $d = n$, we may exactly reproduce $(A,B,C,D)$ in our first stage, and therefore the algorithm reduces to applying the HNA algorithm on $(A,B,C,D)$. In this case, we are guaranteed to obtain an accurate reduced system~\cite{glover1984optimal}, but the computational cost is often prohibitive, requiring $\mathcal{O}(n^3)$ operations~\cite{antoulas2000survey}. On the other hand, setting $d = k$ allows us to complete the model reduction task within the first stage using only $\mathcal{O}(Nk^3)$ operations~\cite{nakatsukasa2018AAA}. However, the error made by this reduced system is entirely controlled by the approximation error of $G - R_k$, which is not guaranteed to be optimal. Our two-stage algorithm displays both accuracy and efficiency by allowing us to choose the degree $d$ adaptively so that we have control on $\norm{G-R_d}_\infty$ (see~\cref{sec:AAK} for the definition of $\norm{\cdot}_\infty$), leaving the task of further reducing the system to the desired rank $k$ to the HNA algorithm, requiring $\mathcal{O}(d^3)$ operations. 

While any combination of a moment-matching algorithm and an SVD-based algorithm leads to a two-stage algorithm, we advocate for using the block-AAA and HNA algorithms. Such a combination allows us to adapt our algorithm to model reduction tasks that use different measurements of accuracy (see~\cref{sec:adaptation}), enhancing our algorithm's flexibility. However, both algorithms need to be modified to be numerically stable. For the output of the AAA algorithm to be useful for the HNA algorithm, we must find a system realization of the rational approximation. Whereas the system generated by the canonical block-AAA algorithm~\cite{aumann2023practical,gosea2021algorithms} causes numerical issues in the second stage (see~\cref{sec:stableAAA}), we propose a modification of the block-AAA algorithm that sacrifices some accuracy for numerical stability. In addition, the HNA formulas proposed in~\cite{glover1984optimal} are algebraic and can lead to an ill-conditioned output system $(\hat{A},\hat{B},\hat{C},\hat{D})$ (see~\cref{sec:stabilizetheta}). To promote the robustness of our two-stage algorithm, we also modify the HNA algorithm and provide a careful analysis to show that our modification resolves numerical issues.

The paper is organized as follows. In~\cref{sec:prelim}, we review background information and introduce notation. In~\cref{sec:algorithm}, we describe our two-stage algorithm for model reduction problems in the Hankel norm, followed by a careful analysis of the algorithm in~\cref{sec:analysis}. We adapt the algorithm to the model reduction problem in the $H^\infty$ norm in~\cref{sec:adaptation} and discuss the choice of our algorithm's parameters in~\cref{sec:parameters}. Finally, in~\cref{sec:expdis}, we conclude with numerical experiments on real-world datasets.

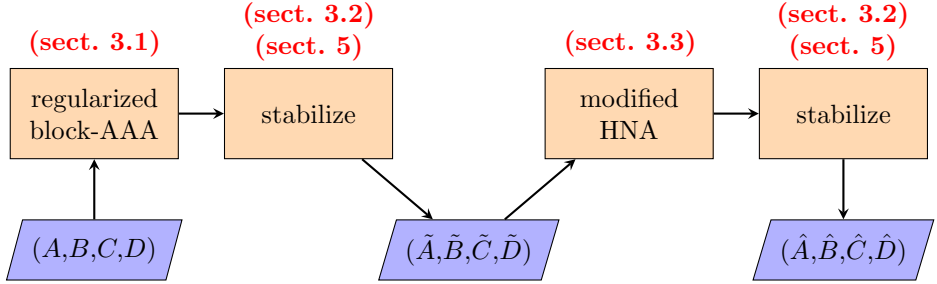
\begin{figure}
\tikzset{node distance = 0.8cm and 0.6cm}
\begin{tikzpicture}

\node (in1) [io] { $(A,\!B,\!C,\!D)$};
\node (pro1) [process, above= of in1] { regularized block-AAA};
\node (pro2) [process, right= of pro1] {stabilize};
\node (in2) [io, below right= of pro2] {$(\tilde{A},\!\tilde{B},\!\tilde{C},\!\tilde{D})$};
\node (pro3) [process, above right= of in2] {modified HNA};
\node (pro4) [process, right= of pro3] {stabilize};
\node (in3) [io, below= of pro4] {$(\hat{A},\!\hat{B},\!\hat{C},\!\hat{D})$};

\node at (0,2.75) 
    {\textcolor{red}{\textbf{(sect.~\ref{sec:stableAAA})}}};
    
\node at (2.85,3.12) 
    {\textcolor{red}{\textbf{(sect.~\ref{sec:stabilize})}}};
    
\node at (2.85,2.7) 
    {\textcolor{red}{\textbf{(sect.~\ref{sec:adaptation})}}};
    
\node at (7.12,2.75) 
    {\textcolor{red}{\textbf{(sect.~\ref{sec:stabilizetheta})}}};
    
\node at (9.95,3.12) 
    {\textcolor{red}{\textbf{(sect.~\ref{sec:stabilize})}}};
    
\node at (9.95,2.7) 
    {\textcolor{red}{\textbf{(sect.~\ref{sec:adaptation})}}};

\draw [arrow] (in1) -- (pro1);
\draw [arrow] (pro1) -- (pro2);
\draw [arrow] (pro2) -- (in2);
\draw [arrow] (in2) -- (pro3);
\draw [arrow] (pro3) -- (pro4);
\draw [arrow] (pro4) -- (in3);


\end{tikzpicture}

\caption{A flowchart of our two-stage algorithm.}
\label{fig:flowchart}
\end{figure}

\section{Preliminaries}\label{sec:prelim}

In this section, we review some background information.

\subsection{The Lyapunov equations, controllability, and observability}

Let $(A,B,C,D)$ be a stable LTI system defined in~\cref{eq.LTIDS}. Let $P$ and $Q$ be Hermitian, positive semi-definite matrices that satisfy the following Lyapunov equations:
\begin{equation}\label{eq.lyapunov}
	AP + PA^* + BB^* = 0, \qquad A^*Q + QA + C^*C = 0.
\end{equation}
The matrices $P$ and $Q$ can be explicitly expressed by the following matrix integrals:
\begin{equation}\label{eq.PQ}
	\begin{aligned}
		P = \int_0^\infty \text{exp}(At) BB^* \text{exp}(A^* t) dt, \qquad Q = \int_0^\infty \text{exp}(A^*t) C^*C \text{exp}(A t) dt.
	\end{aligned}
\end{equation}
The Lyapunov equations are important in understanding the controllability and observability of the LTI system. In particular, the matrix $P$ is called the controllability Gramian and is positive definite if and only if the system~\cref{eq.LTIDS} is controllable. That is, for any $x_0, x_1 \in \C^n$ and any $T > 0$, there exists an input $u$ on $[0,T]$ that makes $x(T) = x_1$ when $x(0) = x_0$. Likewise, $Q$ is called the observability Gramian and is positive definite if and only if the system~\cref{eq.LTIDS} is observable. That is, for any $T > 0$, the initial state $x(0)$ can be determined by the input $u$ and the output $y$ on $[0,T]$~\cite{zhou1998essentials}. The Lyapunov equations motivate SVD-based model reduction algorithms and are important in analyzing our modified HNA algorithm (see~\cref{sec:analysis}).

\subsection{The Hankel operator} Associated with the matrices $A$, $B$, and $C$ in~\cref{eq.lyapunov} is the Hankel operator $\Gamma_{(A,B,C)}: L^2((0,\infty),\C^{m}) \rightarrow L^2((0,\infty), \C^{p})$ defined by
\begin{equation}\label{eq.hankelop}
	(\Gamma_{(A,B,C)} f)(t) = \int_0^\infty C \text{exp}(A(t+s)) Bf(s) ds.
\end{equation}
The Hankel operator maps the past inputs to the future outputs with the initial state $x(0)$~\cite{glover1984optimal}. To see this, assume $u(t)$, $x(t)$, and $y(t)$ form a solution to the system~\cref{eq.LTIDS}, where $u(t) = 0$ for $t \geq 0$ (i.e., no inputs are given in the future). Define $v(t) = u(-t)$ as the inputs given backwards in time. Then, we have $(\Gamma_{(A,B,C)} v)(t) = y(t)$ for all $t \geq 0$. Motivated by this physical interpretation, we propose the following problem of model reduction in the Hankel norm.
\begin{problem}\label{prob.MRHankel}
Find a stable reduced system $(\hat{A},\hat{B},\hat{C},\hat{D})$ to~\cref{eq.LTIDS} with the transfer function $\hat{G}$, where $\hat{A} \in \C^{k \times k}, \hat{B} \in \C^{k \times m}, \hat{C} \in \C^{p \times k}$, and $\hat{D} \in \C^{p \times m}$ that minimizes the operator norm $\big\|{\Gamma_{({A},{B},{C})} - \Gamma_{(\hat{A},\hat{B},\hat{C})}}\big\|_{L^2 \rightarrow L^2}$.
\end{problem}
The fact that the Hankel operators are close to each other means that the systems on the same past inputs give close future outputs. To show that the problem is well-posed, we note that the Hankel operator $\Gamma_{(A,B,C)}$ is a bounded linear operator between two separable Hilbert spaces. It has a finite rank of $\leq n$, i.e., it can be represented as 
\[
	\Gamma_{(A,B,C)} f = \sum_{j=1}^n \sigma_j \langle f, \phi_j \rangle_{L^2} \psi_j, \qquad f \in L^2((0,\infty),\C^m),
\]
where $\{\phi_j\}_{j=1}^n$ and $\{\psi_j\}_{j=1}^n$ are orthonormal sets in $L^2((0,\infty),\!\C^m)$ and $L^2((0,\infty),\!\C^p)$, respectively, and $\sigma_1 \geq \cdots \geq \sigma_n \geq 0$ are called the Hankel singular values of $\Gamma_{(A,B,C)}$. It turns out that $\sigma_1, \ldots, \sigma_n$ are exactly the singular values of $PQ$, where $P$ and $Q$ are given in~\cref{eq.PQ}. As we show next, the Hankel singular values are related to the theory of low-rank Hankel approximations of the Hankel operator.

\subsection{The Adamyan--Arov--Krein theory}\label{sec:AAK}

Given the Hankel operator $\Gamma_{\!(A,B,C)}$ defined in~\cref{eq.hankelop}, for any $0 \leq k \leq n-1$, we have $\min_{\Gamma} \norm{\Gamma_{(A,B,C)} - \Gamma}_{L^2 \rightarrow L^2} = \sigma_{k+1}$, where $\Gamma$ ranges over all bounded linear operators of rank $\leq k$. However, this does not solve Problem~\ref{prob.MRHankel} because the optimal $\Gamma$ may not necessarily be a Hankel operator. Adamyan, Arov, and Krein studied optimal Hankel approximation in the 1960s and 70s, and their main results are commonly referred to as AAK theory~\cite{AAK2,AAK1}. To introduce AAK theory, we define the following norm, denoted by $\norm{\cdot}_\infty$, for matrix-valued functions on the imaginary axis:
\begin{equation}\label{eq.Hinfnorm}
	\norm{f}_\infty = \text{ess\,sup}_{\text{Re}(z) = 0} \norm{f(z)}_2,
\end{equation}
where $\norm{\cdot}_2$ is the matrix $2$-norm. Let $H^\infty$ be the Hardy space that contains all functions $h\!: \C \!\rightarrow\! \C^{p \times m}$ that are analytic in the open left half-plane and whose non-tangential limits on the imaginary axis are bounded in the $\norm{\cdot}_\infty$ norm. AAK theory states the following:

\begin{theorem}[Adamyan--Arov--Krein]\label{thm.AAK} 
Using the notation defined in this section, for any fixed $0 \leq k \leq n-1$, there exists $\hat{A} \in \C^{k \times k}$, $\hat{B} \in \C^{k \times m}$, $\hat{C} \in \C^{p \times k}$, and $\hat{D} \in \C^{p \times m}$ such that the transfer function of the reduced system $\hat{G}(z) = \hat{D} + \hat{C}(zI - \hat{A})^{-1} \hat{B}$ satisfies
\begin{equation}
	\sigma_{k+1} = \inf_{F \in H^\infty} \big\|{G - \hat{G} - F}\big\|_{\infty},
\end{equation}
where we take $\hat{G}(z) \equiv 0$ when $k = 0$.
\end{theorem}

\begin{proof}
See~\cite[Thm.~0.1]{AAK1}.
\end{proof}

In particular, applying~\Cref{thm.AAK} with $k=0$, we have that
\[
	\sigma_1 = \norm{\Gamma_{(A,B,C)}}_{L^2 \rightarrow L^2} = \inf_{F \in H^\infty} \norm{G - F}_{\infty},
\]
i.e., the norm of a Hankel operator is equivalent to the $\norm{\cdot}_\infty$ distance of its transfer function to the Hardy space $H^\infty$. If we identify the transfer function $G$ with the class of functions $\{G \!+\! F \mid F \!\in\! H^\infty\}$, which is an equivalence class under the equivalence relation $G_1 \!\sim\! G_2 \Leftrightarrow G_1\!-\!G_2 \!\in\! H^\infty$, we can define a norm on the equivalence classes:
\begin{equation}\label{eq.hankelnorm}
	\norm{G}_H = \inf_{F \in H^\infty} \norm{G-F}_\infty.
\end{equation}
The norm in~\cref{eq.hankelnorm} is equal to the operator norm of $\Gamma_{(A,B,C)}$. Hence, the goal of Problem~\ref{prob.MRHankel} is equivalent to minimizing $\big\|{G - \hat{G}}\big\|_H$. We can assume, without loss of generality, that $p \leq m$; otherwise, we can solve the model reduction problem on the dual system $(A^*,C^*,B^*,D^*)$ with the transfer function $G^*$ and obtain the reduced dual system $(\hat{A}^*,\hat{C}^*,\hat{B}^*,\hat{D}^*)$ with the transfer function $\hat{G}^*$ with $\big\|{{G} - \hat{G}}\big\|_H = \big\|{G^* - \hat{G}^*}\big\|_H$. As the degree of the rational approximation $d$ in our two-stage algorithm increases, the Hankel error of the reduced system, i.e., $\big\|{G-\hat{G}}\big\|_H$, approaches $\sigma_{k+1}$, which is the theoretically optimal Hankel error bound of a rank-$k$ Hankel approximation.

\subsection{The adaptive Antoulas--Anderson (AAA) algorithm} The AAA algorithm was first proposed in~\cite{nakatsukasa2018AAA}. It provides a fast and empirically robust way to construct rational approximations of a scalar-valued function. The algorithm was later generalized to matrix-valued functions~\cite{gosea2021algorithms,lietaert2022automatic}. Given a set of samples $\{(z_i,G(z_i))\}_{i=1}^N$ of $G: \Omega \rightarrow \C^{p \times m}$ on a domain $\Omega \subset \C$, the block-AAA algorithm~\cite{gosea2021algorithms} computes rational approximations $R_0, R_1, \ldots$ adaptively, where $R_0(z) = 0$ for all $z \in \C$ and\footnote{The original barycentric formula proposed in~\cite{gosea2021algorithms} is slightly different. The form in~\cref{eq.bary1} is suggested in~\cite{aumann2023practical} to enable one to find a system realization of $R_d$ (see~\cref{sec:stableAAA}).}
\begin{equation}\label{eq.bary1}
	R_d(z) = \left(I + \sum_{j=1}^d \frac{W^{(d)}_j}{z-z_{i_j}} \right)^{-1} \left(\sum_{j=1}^d \frac{W^{(d)}_jG(z_{i_j})}{z-z_{i_j}}\right), \qquad d \geq 1,
\end{equation}
where $W_{j}^{(d)} \in \C^{p \times p}$ for $1 \leq j \leq d$ are weight matrices. It is comprised of two ideas. First, given $R_{d-1}$, we choose the next support point $z_{i_{d}} \in \{z_i\}_{i=1}^N$ to be at a location that maximizes $\norm{R_{d-1}(z_i) - G(z_i)}_F$ for $1 \leq i \leq N$. To greedily make $R_{d}$ approximate well at $z_{i_d}$, we compute $W^{(d)}_j$ in~\cref{eq.bary1} by solving a least-squares problem
\begin{equation}\label{eq.lsqblockAAA}
	\underbrace{\begin{bmatrix}
		W_1^{(d)}\! & \!\!\!\cdots\!\!\! & \!W_d^{(d)}\!
	\end{bmatrix}}_{W^{(d)}}\!
	\underbrace{
	\begin{bmatrix}
		\frac{G(z_{k_1}) - G(z_{i_1})}{z_{k_1} - z_{i_1}}\! & \!\!\!\cdots\!\!\! & \!\frac{G(z_{k_{N-d}}) - G(z_{i_1})}{z_{k_{N-d}} - z_{i_1}} \\
		\vdots & \!\!\!\ddots\!\!\! & \vdots \\
		\frac{G(z_{k_1}) - G(z_{i_d})}{z_{k_1} - z_{i_d}}\! & \!\!\!\cdots \!\!\!& \!\frac{G(z_{k_{N-d}}) - G(z_{i_d})}{z_{k_{N-d}} - z_{i_d}}
	\end{bmatrix}
	}_{M^{(d)}}
	=
	\underbrace{
	\begin{bmatrix}
		-G(z_{k_1})\! & \!\!\!\cdots\!\!\! & \!-G(z_{k_{N-d}})\!
	\end{bmatrix}
	}_{G^{(d)}},
\end{equation}
where $\{k_1, \ldots, k_{N-d}\} = [N] \setminus \{i_1, \ldots, i_d\}$. As long as the weights $W^{(d)}_j$ are non-singular, $R_{d}$ interpolates $G$ at $z_{i_{d}}$. The block-AAA algorithm is important in our algorithmic design because it serves as an efficient way to reduce the system to a rank that is not so large that the (modified) HNA algorithm in our second stage is computationally prohibitive, and not so small that the approximation accuracy of the intermediate system is poor. We show, in particular, that when using our two-stage algorithm to solve Problem~\ref{prob.MRHankel}, the value of $\big\|{G - \hat{G}}\big\|_H - \sigma_{k+1}$ is roughly bounded by twice the block-AAA approximation error given by $\norm{G - R_d}_\infty$ (see~\cref{sec:analysis}).

\subsection{Balanced realizations and Hankel norm approximations}\label{sec:BRHNA}
Given a stable system $(A,B,C,D)$ and assume that for some $r \geq 1$, its Hankel singular values satisfy ${\sigma}_1 \geq \cdots \geq {\sigma}_{k} > {\sigma}_{k+1} = \cdots = {\sigma}_{k+r} > {\sigma}_{k+r+1} \geq \cdots \geq {\sigma}_n$. One can use the SVD to compute a balanced realization $(A_b,B_b,C_b,D_b)$~\cite{laub1987computation} with the associated matrices $P_b = Q_b = \diag(\sigma_1, \ldots, \sigma_{k}, \sigma_{k+r+1}, \ldots, \sigma_n, \sigma_{k+1}, \ldots, \sigma_{k+r})$. The matrices in the balanced realization satisfy the Lyapunov equations given by
\begin{equation*}
	A_bP_b + P_bA_b^* + B_bB_b^* = 0, \qquad A_b^*Q_b + Q_bA_b + C_b^*C_b = 0,
\end{equation*}
such that the transfer function $G_b(z) = D_b + C_b(zI - A_b)^{-1}B_b = G(z)$. Suppose we partition the matrices so that
\begin{equation}\label{eq.decomposeABC}
	{A}_b = \left[
	\begin{array}{c|c}
		\!\!{A}_{11} & {A}_{12}\!\! \\
		\hline
		\!\!{A}_{21} & {A}_{22}\!\!
	\end{array}
	\right], \;
	{B}_b = \left[
	\begin{array}{c}
		\!{B}_{1}\! \\
		\hline
		\!{B}_{2}\!
	\end{array}
	\right], \;
	{C}_b = \left[
	\begin{array}{c|c}
		\!\!{C}_{1} & {C}_{2}\!\!
	\end{array}
	\right], \;
	P_b = 
	\left[
	\begin{array}{c|c}
		{\Sigma}_1 & 0 \\
		\hline
		0 & {\sigma}_{k+1} I_r
	\end{array}
	\right],
\end{equation}
where ${\Sigma}_1 = \diag({\sigma}_1, \ldots, {\sigma}_{k}, {\sigma}_{k+r+1}, \ldots, {\sigma}_n)$, ${A}_{11} \in \C^{(n-r) \times (n-r)}$, ${B}_1 \in \C^{(n-r) \times m}$, and ${C}_1 \in \C^{p \times (n-r)}$. The sizes of the other matrices in~\cref{eq.decomposeABC} can be inferred. Glover proposed the following HNA algorithm~\cite{glover1984optimal}:
\begin{equation}\label{eq.gloveralg}
\begin{aligned}
	\hat{A} &= ({\Sigma}_1^2 - {\sigma}_{k+1}^2 I)^{-1} ({\sigma}_{k+1}^2 {A}_{11}^* + {\Sigma}_1 {A}_{11} \tilde{\Sigma}_1 - {\sigma}_{k+1} {C}_1^* U {B}_1^*), \\
    \hat{B} &= ({\Sigma}_1^2 - {\sigma}_{k+1}^2 I)^{-1} ({\Sigma}_1 {B}_1 + {\sigma}_{k+1} {C}_1^* U), \\
    \hat{C} &= {C}_1 {\Sigma}_1 + {\sigma}_{k+1} U {B}_1^*, \\
    \hat{D} &= {D} - {\sigma}_{k+1}U,
\end{aligned}
\end{equation}
where $U = -C_2 (B_2^*)^\dagger$. We denote~\cref{eq.gloveralg} by $\texttt{HNA}(A_{ij},B_i,C_j,D,{\Sigma}_1,{\sigma}_{k+1},U)$. The system $(\hat{A},\hat{B},\hat{C},\hat{D})$ is unstable so it is not a solution to Problem~\ref{prob.MRHankel}. Nevertheless, the transfer function defined by $\hat{G}(z) = \hat{D} + \hat{C}(zI-\hat{A})^{-1}\hat{B}$ satisfies $\big\|{G - \hat{G}}\big\|_\infty = \big\|{G-\hat{G}}\big\|_H = \sigma_{k+1}$~\cite{glover1984optimal}. Moreover, $\hat{G}$ can be written as the sum of a transfer function $\hat{G}^s$ of a stable rank-$k$ system and an $H^\infty$ function $\hat{G}^u$. By~\Cref{thm.AAK}, we have $\big\|{G-\hat{G}^s}\big\|_H = \big\|{G-\hat{G}}\big\|_H = \sigma_{k+1}$ and thus $\hat{G}^s$ is an optimal solution to Problem~\ref{prob.MRHankel}.

The HNA algorithm extends the so-called suboptimal HNA algorithm~\cite{antoulas2005approximation}, both of which produce a reduced system with $(G(z) - \hat{G}(z))^*(G(z)-\hat{G}(z)) = cI$ for every $z$ on the imaginary axis, where $c$ is a constant close to $\sigma_{k+1}^2$. Such a reduced system is called an all-pass dilation, and it is guaranteed to be equivalent to a rank-$k$ system in the Hankel norm (see~\cref{sec:AAK}). Note that in the HNA algorithm, the condition ${\sigma}_{k} > {\sigma}_{k+1} = \cdots = {\sigma}_{k+r} > {\sigma}_{k+r+1}$ is algebraic in nature and numerical issues arise when $\sigma_{k}-\sigma_{k+1}$ or $\sigma_{k+r} - \sigma_{k+r+1}$ is small (see~\cref{sec:stabilizetheta}). We propose a choice of ${\Sigma}_1$ to remedy the numerical issue. While the reduced system from our modified HNA algorithm is no longer an all-pass dilation, the subsequent analysis shows that the reduced system remains almost optimal with this modification. As a result, the modified HNA algorithm becomes a powerful tool to reduce the intermediate system generated by the block-AAA algorithm to the desired rank.

\section{Our two-stage algorithm for model reduction}\label{sec:algorithm}
This section combines the block-AAA and HNA algorithms to obtain an efficient, robust, and flexible two-stage algorithm for reducing the LTI system in~\cref{eq.LTIDS} given a set of samples on the imaginary axis from the transfer function in~\cref{eq.transfer} (see~\Cref{alg:hankelAAA}). The output of \Cref{alg:hankelAAA} aims to solve Problem~\ref{prob.MRHankel}. Later, in~\cref{sec:adaptation}, we adapt our algorithm to solve a different model reduction problem.

\begin{algorithm}
	\caption{Computing a ROM with the block-AAA algorithm and the HNA algorithm}
	\label{alg:hankelAAA}
	\hspace*{\algorithmicindent} \textbf{Input} A set of samples $\{(z_i,G(z_i))\}_{i=1}^N$ from the transfer function in~\cref{eq.transfer}, where $\text{Re}(z_i) = 0$ for $1 \leq i \leq N$, a target rank $k$, a regularization parameter $\lambda \!\geq\! 0$, a tolerance $\epsilon \!\geq\! 0$, and a threshold $\gamma \!\geq\! 0$.\\\textbf{Stage I: The block-AAA algorithm}
	\begin{algorithmic}[1]
	\STATE\label{state.startAAA} Construct a degree-$d$ rational approximation $R_d$ from $\{(z_i,G(z_i))\}_{i=1}^N$ using the block-AAA algorithm with parameter $\lambda$ and find a system $(\tilde{A},\tilde{B},\tilde{C},\tilde{D})$ of rank $K \leq pd$ whose transfer function is $R_d$ (see~\cref{sec:stableAAA})
	\STATE\label{state.endAAA} Stabilize the system $(\tilde{A},\tilde{B},\tilde{C},\tilde{D})$ (see~\cref{sec:stabilize})
	\end{algorithmic}
	\textbf{Stage II: The modified HNA algorithm}
	\begin{algorithmic}[1]
	\addtocounter{ALC@line}{2}
	\STATE\label{state.startHNA} Compute a balanced realization $(\tilde{A},\tilde{B},\tilde{C},\tilde{D})$ of the system to obtain the Hankel singular values $\tilde{\sigma}_1, \ldots, \tilde{\sigma}_K$, where $K$ is the rank of the system
	\STATE\label{state.decomposemats} Find all Hankel singular values within the distance of $\epsilon$ to $\tilde{\sigma}_{k+1}$ and construct $\tilde{A}_{ij}, \tilde{B}_i, \tilde{C}_j, \tilde{\Sigma}_i$ for $i,j = 1, 2$ and $U$ using $\gamma$ (see~\cref{sec:stabilizetheta})
	\STATE Construct $(\hat{A}, \hat{B}, \hat{C}, \hat{D})$ using $\texttt{HNA}(\tilde{A}_{ij},\tilde{B}_i,\tilde{C}_j,\tilde{D},\tilde{\Sigma}_1,\tilde{\sigma}_{k+1},U)$
	\STATE\label{state.endHNA} Find a stable system of rank $k$ using $(\hat{A}, \hat{B}, \hat{C}, \hat{D})$ (see~\cref{sec:stabilize})
	\end{algorithmic}
\end{algorithm}

\comment{
\begin{algorithm}
	\caption{Construction of $U$}
	\label{alg:constructU}
	\hspace*{\algorithmicindent} \textbf{Input} $\tilde{B}_2, \tilde{C}_2, \gamma$
	\begin{algorithmic}[2]
	\STATE Compute the SVD of $\tilde{C}_2^* = U_CS_CV_C^*$
	\STATE $q \gets \min \{j \mid (S_C)_{j+1,j+1} \leq \gamma\}$
	\STATE $U^{(1)}_C \gets (U_C)_{:,1:q}$, $S^{(1)}_C \gets (S_C)_{1:q,1:q}$ \hfill\COMMENT{separate out small singular values}
	\STATE Find the least-squares solution of $\big(U^{(1)}_C S^{(1)}_C\big) \big(V^{(1)}_B \big)^*= \tilde{B}_2$ 
	\STATE\label{state.append} Construct $V_B \in \R^{p \times m}$ by appending $p-q$ arbitrary orthonormal columns to $V^{(1)}_B$ that are also orthogonal to the column space of $V^{(1)}_B$
	\STATE $U \gets V_CV_B^*$
	\end{algorithmic}
\end{algorithm}
}

\subsection{Finding a stable realization of a rational approximation}\label{sec:stableAAA}

Applying the block-AAA algorithm gives us a rational approximation $R_d$ of $G$ in the form of~\cref{eq.bary1}. To apply the subsequent HNA algorithm to this intermediate system, we must explicitly find the associated LTI system $(\tilde A, \tilde B, \tilde C, \tilde D)$ that corresponds to the transfer function $R_d$. In~\cite{aumann2023practical}, it is shown that $R_d$ can be realized by
\begin{equation*}
	\begin{aligned}
		\tilde A = \diag(z_{i_1}, \ldots, z_{i_d}) \!\otimes\! I_p \!-\! (\mathbbm{1}_d \!\otimes\! I_p)\tilde C, \quad \tilde B = \begin{bmatrix}
			G(z_{i_1})^* & \!\!\cdots\!\! & G(z_{i_d})^*
		\end{bmatrix}^*, \quad \tilde C = W^{(d)}
	\end{aligned},
\end{equation*}
where $W^{(d)}$ is defined in~\cref{eq.lsqblockAAA}. Unfortunately, as $d$ increases, the matrix $M^{(d)}$ in~\cref{eq.lsqblockAAA} becomes ill-conditioned in practice, resulting in very large entries in $W^{(d)}$. Therefore, the first term $\diag(z_{i_1}, \ldots, z_{i_d}) \otimes I_p$ of $\tilde{A}$ has a much smaller order of magnitude than that of the second term $(\mathbbm{1}_d \!\otimes\! I_p)\tilde C$. This makes the numerical computation with $\tilde{A}$ difficult because $\diag(z_{i_1}, \ldots, z_{i_d}) \otimes I_p$ is easily contaminated by numerical errors. To combat this potential failure, we propose to solve the following regularized least-squares problem instead of~\cref{eq.lsqblockAAA}:
\begin{equation}\label{eq.regularizedAAA}
	W^{(d)} = \text{arg\,min}_W \norm{G^{(d)} - WM^{(d)}}_F^2 + \lambda \norm{W}_F^2,
\end{equation}
where $\lambda \geq 0$ is a regularization parameter. When $\lambda = 0$, we reproduce the least-squares problem in~\cref{eq.lsqblockAAA}. As $\lambda$ increases, the magnitudes of entries in $W$ become smaller, making the computation with $\tilde{A}$ more stable, whereas $R_d$ may become a less accurate approximation of $G$. In~\cref{sec:parameters}, we discuss the choice of $\lambda$, and in~\cref{sec:expdis}, we test the effect of the regularization parameter $\lambda$.

\subsection{Stabilizing an LTI system while preserving the Hankel operator}\label{sec:stabilize}

The system $(\tilde A, \tilde B, \tilde C, \tilde D)$ in~\cref{sec:stableAAA} is unstable in general. To apply the modified {HNA} algorithm (see~\cref{sec:stabilizetheta}), we want to work with a stable system. We can stabilize a general LTI system $(A,B,C,D)$ without changing its Hankel operator if $A \in \C^{n \times n}$ is diagonalizable with no purely imaginary eigenvalue.

Suppose we diagonalize ${A} = {X} {\Lambda} {X}^{-1}$, where ${\Lambda} = \diag({\lambda}_1, \ldots, {\lambda}_\iota, {\lambda}_{\iota+1}, \ldots, {\lambda}_{n-r})$ is a diagonal matrix with $\text{Re}({\lambda}_i) < 0$ for $1 \leq i \leq \iota$ and $\text{Re}({\lambda}_i) > 0$ for $i \geq \iota + 1$. If we define ${A}_s = \diag({\lambda}_1, \ldots, {\lambda}_{\iota})$, ${B}_s$ and $B_u$ to be the first $\iota$ and the last $n-r-\iota$ rows of $X^{-1}{B}$, respectively, and ${C}_s$ and $C_u$ to be the first $\iota$ and the last $n-r-\iota$ columns of ${C}X$, respectively, then we have
\begin{equation*}
	\begin{aligned}
	{G}^s(z) &:= {D} + {C}_s(zI - {A}_s)^{-1}{B}_s = {D} + {C}X(zI - {\Lambda})^{-1}X^{-1}{B} \\
	&\qquad- C_u(zI - \diag({\lambda}_{\iota+1},\ldots,{\lambda}_{n-r}))^{-1}B_u.
	\end{aligned}
\end{equation*}
The map $z \mapsto C_u(zI - \diag({\lambda}_{\iota+1},\ldots,{\lambda}_{n-r}))^{-1}B_u$ is in $H^\infty$ because ${\lambda}_{\iota+1}, \ldots, {\lambda}_{n-r}$ are in the right half-plane. Hence, $(A, B, C, D)$ and $(A_s, B_s, C_s, D)$ define the same Hankel operator. We use this stabilization procedure on both the intermediate system $(\tilde{A},\tilde{B},\tilde{C},\tilde{D})$ and the final reduced system $(\hat{A},\hat{B},\hat{C},\hat{D})$.  

\subsection{The modified HNA algorithm}\label{sec:stabilizetheta}

Let $\tilde{\sigma}_1, \ldots, \tilde{\sigma}_K$ be the Hankel singular values of the intermediate system $(\tilde{A},\tilde{B},\tilde{C},\tilde{D})$. The HNA algorithm requires us to separate all the singular values equal to $\tilde{\sigma}_{k+1}$. In practical applications, we rarely encounter two singular values of a system that are the same. Nevertheless, it is common for two singular values to be close to each other. In particular, if $\tilde{\sigma}_{k} - \tilde{\sigma}_{k+1}$ or $\tilde{\sigma}_{k+1} - \tilde{\sigma}_{k+2}$ are small, then the matrix $\tilde{\Sigma}_1^2 - \tilde{\sigma}_{k+1}^2 I$, which appears in the HNA formulas in~\cref{eq.gloveralg}, becomes ill-conditioned. 
While this does not affect the (relative) accuracy of the computed entries of $\hat{A}$ and $\hat{B}$ because $\tilde{\Sigma}_1^2 - \hat{\sigma}^2 I$ is a diagonal matrix, it usually makes $\hat{A}$ and $\hat{B}$ have large entries. Consequently, the operator norm of the Hankel operator of $(\hat{A},\hat{B},\hat{C},\hat{D})$ is much smaller than the magnitudes of the matrix entries. Such a system is numerically challenging in floating point arithmetic because there can be a significant loss of numerical precision due to cancellation errors (see~\cref{sec:expstability} for an example). This numerical issue is also witnessed by the suboptimal HNA algorithm~\cite[sect.~8.4.4]{antoulas2005approximation}.

To resolve this issue, we propose the modified HNA algorithm. Let $\tilde{\sigma}_1 \geq \cdots \geq \tilde{\sigma}_K \geq 0$ be the Hankel singular values of the intermediate system and let $j_1 \leq k+1 
\leq j_2$ be two indices such that $\tilde{\sigma}_{j_1}, \ldots, \tilde{\sigma}_{k+1}, \ldots, \tilde{\sigma}_{j_2}$ are the Hankel singular values $\tilde{\sigma}_j$ such that $|\tilde{\sigma}_j - \tilde{\sigma}_{k+1}| \leq \epsilon$, and let $r = j_2-j_1+1$ be the number of such singular values. Assume the system $(\tilde{A},\tilde{B},\tilde{C},\tilde{D})$ and the Gramians $\tilde P = \tilde Q = {\diag}(\tilde{\sigma}_1, \ldots, \tilde{\sigma}_{j_1-1},\tilde{\sigma}_{j_2+1}, \ldots, \tilde{\sigma}_{K}, \tilde{\sigma}_{j_1}, \ldots, \tilde{\sigma}_{j_2})$ are the balanced realization of ${R}^s_d$ satisfying the Lyapunov equations in~\cref{eq.lyapunov}. We define
\begin{equation}\label{eq.modifiedpartition}
	\tilde{A} \!=\! \left[
	\begin{array}{c|c}
	\!\!\tilde A_{11} & \tilde A_{12}\!\! \\
	\hline
	\!\!\tilde A_{21} & \tilde A_{22}\!\!
	\end{array}
	\right], 
	\quad
	\tilde{B} \!=\! \left[
	\begin{array}{c}
	\!\tilde B_{1}\! \\
	\hline
	\!\tilde B_{2}\!
	\end{array}
	\right], 
	\quad
	\tilde{C} \!=\! \left[
	\begin{array}{c|c}
	\!\!\tilde C_{1} & \tilde C_{2} \!\!
	\end{array}
	\right],
	\quad
	\tilde P \!=\! \tilde Q \!=\! \left[
	\begin{array}{c|c}
	\!\!\tilde \Sigma_1 & 0\!\! \\
	\hline
	\!\!0 & \tilde \Sigma_2\!\!
	\end{array}
	\right],
\end{equation}
where $\tilde{A}_{11} \!\in\! \C^{(K-r) \times (K-r)}$, $\tilde{B}_{1} \!\in\! \C^{(K-r) \times m}$, $\tilde{C}_{1} \!\in\! \C^{p \times (K-r)}$, and $\tilde{\Sigma}_1 \!\in\! \C^{(K-r) \times (K-r)}$. By setting $\hat{\sigma} = \tilde{\sigma}_{k+1}$, we are guaranteed that multiplying $(\tilde{\Sigma}_1^2 - \hat{\sigma}^2I)^{-1}$ to a matrix increases the magnitudes of its entries by at most a factor of $(\hat{\sigma}\epsilon)^{-1}$.

Next, we explain the construction of the matrix $U$ appearing in the modified HNA algorithm (see~\Cref{alg:hankelAAA}). Let $\tilde C_2^* = U_C S_C V_C^*$ be the SVD of $\tilde C_2^*$. We partition $U_C$ and $S_C$ so that
\[
     U_C =
    \left[
    \begin{array}{c|c}
         U_C^{(1)} & U_C^{(2)}
    \end{array}
    \right]
    \in \C^{r \times r}, \quad S_C = 
    \left[
    \begin{array}{c|c}
         S_C^{(1)} & 0 \\
         \hline
         0 & S_C^{(2)}
    \end{array}
    \right]
    \in \C^{r \times p},
\]
where $S_C^{(1)} \in \R^{q \times q}$, $U_C^{(1)} \in \C^{r \times q}$, and $1 \leq q \leq r$ is chosen to be the maximum number such that $\big\|{S_{C}^{(2)}}\big\| \geq \gamma$. Let $(V_B^{(1)})^* \in \R^{q \times m}$ be the least-squares solution of $(U_C^{(1)} S_C^{(1)}) X = -\tilde B_2$, i.e.,
\begin{equation}\label{eq.constructVB}
    (V_B^{(1)})^* \!=\! -\big((S_C^{(1)})^*(U_C^{(1)})^*U_C^{(1)}S_C^{(1)}\big)^{-1} \!(S_C^{(1)})^*(U_C^{(1)})^* \tilde B_2 \!=\! -(S_C^{(1)})^{-1} \!(U_C^{(1)})^* \tilde B_2.
\end{equation}
Then, we construct $V_B \in \R^{p \times m}$ by appending $p-q$ arbitrary orthonormal columns to $V_B^{(1)}$ that are also orthogonal to the column space of $V_B^{(1)}$. We denote the appended $(p-q) \times m$ matrix by $V_B^{(2)}$. Finally, we set $U = V_C V_B^*$. This construction guarantees that both $\|I - UU^*\|$ and $\|\tilde{B}_2 + \tilde{C}_2^*U\|$ are small (see~\Cref{thm.QEQU}), which is crucial to analyzing the modified HNA algorithm (see~\cref{sec:analysis}).

The modified HNA algorithm is almost identical to the original one, except that $\tilde{\Sigma}_2$ is no longer a multiple of the identity matrix. Consequently, Glover's results on the optimal Hankel approximation (see~\cref{sec:BRHNA}) no longer apply. In~\cref{sec:analysis}, we analyze the effect of $\epsilon$ on the rank and the approximation accuracy of the reduced order model to justify our modification.

\subsection{An example: system identification using our two-stage algorithm}\label{sec:exphilbert}

We now present a synthetic example to illustrate our two-stage algorithm (see~\cref{sec:expdis} for applications). Consider the case where we cannot access the matrices $A, B, C$, and $D$ that define the LTI system. Instead, we have a protocol to evaluate the transfer function $G$ at a point along the imaginary axis. The task of system identification/approximation is to find a system $(\hat A, \hat B, \hat C, \hat D)$, possibly with a prescribed rank, whose transfer function is close to $G$ in the $\norm{\cdot}_\infty$ norm or the Hankel norm.

Since our two-stage algorithm requires only a set of samples, it is appropriate for this task. Here, we define  $G(z) = \sum_{j=1}^{500} j^{-1} (1+z)^{-j} (1-z)^j$, which comes from truncating the formal power series that defines the Hilbert operator~\cite{peller2006Hankel}. In this case, the singular values of the associated Hankel operator decay slowly, making the system hard to approximate. We sample $G$ on a grid of over $600$ points $z_i$ placed logarithmically on $[-i,-10^{-4}i]$ and $[10^{-4}i,i]$ and use the two-stage algorithm to compute a rank-$10$ system $(\hat A, \hat B, \hat C, \hat D)$ so that $\hat{G}$ approximates $G$ in the Hankel norm. Given only $\hat{G}$ and $G$, estimating the Hankel error $\big\| \hat{G} - G \big\|_H$ is difficult~\cite{garnett2007bounded}. Hence, we measure the $H^\infty$ error $\big\| \hat{G} - G - \tilde{H} - \hat{H} \big\|_\infty$, where $\tilde{H}$ and $\hat{H}$ are the $H^\infty$ functions coming from stabilizing the intermediate and the final reduced systems, respectively (see~\cref{sec:stabilize}). This $H^\infty$ error is an upper bound of the Hankel error (see~\Cref{thm.AAK}). We apply our two-stage algorithm with different values of $d$, the degree of the rational approximation. As shown in~\Cref{fig:hilbert}, as $d$ increases, the approximation error tends to decrease because the rational approximation tends to be more accurate, and the execution time increases. This shows a trade-off between accuracy and time, which makes our two-stage algorithm flexible: one can select a degree $d$ within the block-AAA algorithm based on a preference for accuracy or efficiency.

\begin{figure}
\centering
\begin{minipage}{.48\textwidth}
\begin{overpic}[width=1\textwidth]{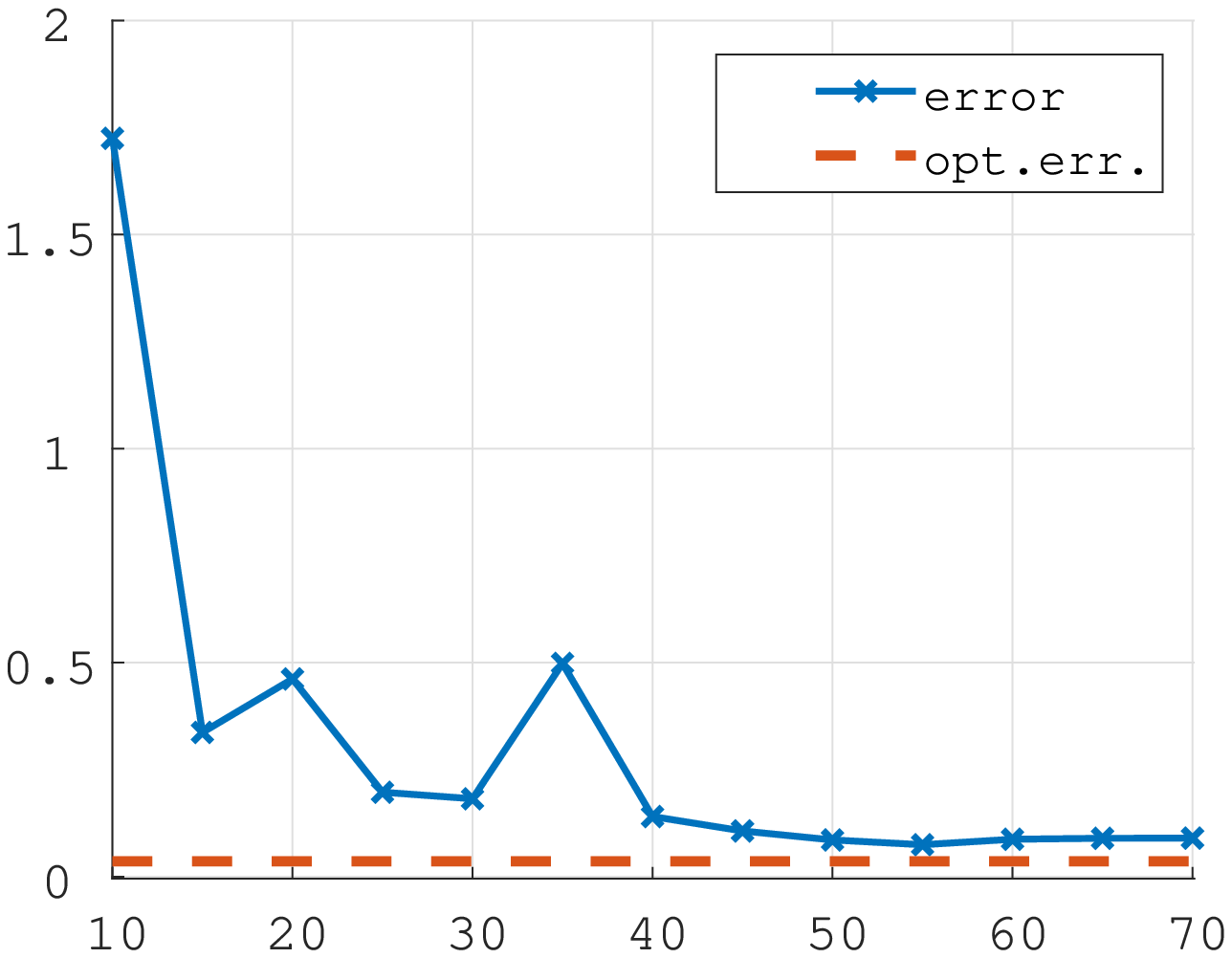}
\put(50,-2) {\footnotesize $d$}
\put(-2,32) {\rotatebox{90}{\footnotesize $\|{\hat{G} - G}\|_H$}}
\end{overpic}
\caption*{\footnotesize (a) Approximation error}
\end{minipage}
\begin{minipage}{.48\textwidth}
\begin{overpic}[width=1\textwidth]{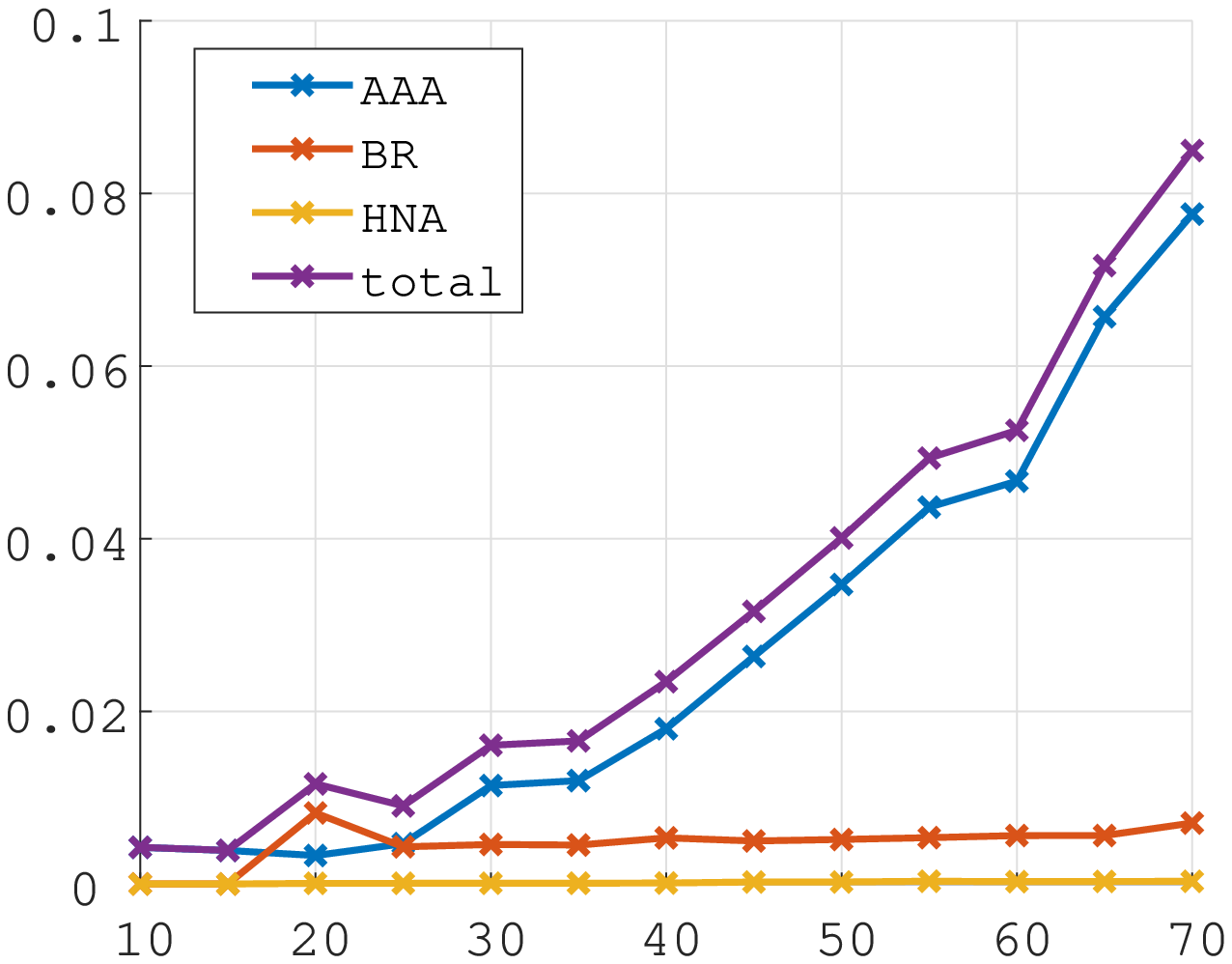}
\put(52,-2) {\footnotesize $d$}
\put(0,37) {\rotatebox{90}{\footnotesize time}}
\end{overpic}
\caption*{\footnotesize (b) Execution time}
\end{minipage}
\caption{Model reduction on the truncated Hilbert example. Left: The approximation error as the degree $d$ increases. Right: The execution time of the block-AAA, balanced realization (BR), and modified HNA algorithms as $d$ increases.}
\label{fig:hilbert}
\end{figure}

\section{Analysis of our two-stage algorithm}\label{sec:analysis}

Let $\hat{G}$ and $\tilde{G}$ be the transfer functions of the final reduced and intermediate systems, respectively. We have
\[
	\big\| \hat{G} - G \big\|_H = \big\| \hat{G} - G \big\|_H \leq \big\| \hat{G} - \tilde{G} \big\|_H + \big\| \tilde{G} - G \big\|_H.
\]
When $\epsilon = 0$ in~\Cref{alg:hankelAAA}, i.e., Glover's original HNA algorithm~\cite{glover1984optimal} is used, we have by Weyl's inequality
\begin{equation}\label{eq.modifiederror}
	\big\| \hat{G} - G \big\|_H \leq \tilde{\sigma}_{k+1} + \big\| \tilde{G} - G \big\|_H \leq \sigma_{k+1} + 2\big\| \tilde{G} - G \big\|_H \leq \sigma_{k+1} + 2\big\| {R}_d - G \big\|_\infty,
\end{equation}
where $\tilde{\sigma}_{k+1}$ and $\sigma_{k+1}$ are the $(k+1)$th Hankel singular value of the Hankel operators associated with $\tilde{G}$ and $G$, respectively.
Since ${\sigma}_{k+1}$ is the optimal Hankel error of a rank-$k$ approximation of $G$ (see~\Cref{thm.AAK}), the Hankel error of $\hat{G}$ is larger than the optimal value by at most twice the rational approximation error $\big\| {R}_d - G \big\|_\infty$. When we select $\epsilon > 0$ (see~\cref{sec:stabilizetheta}), we are no longer guaranteed to have~\cref{eq.modifiederror}. Hence, a careful analysis of the modified HNA algorithm is needed. To do so, we define two non-negative quantities by $Q_E = \norm{\Delta_1}, Q_U = \norm{\Delta_2}$, where
\begin{equation}
	\Delta_1 = \tilde B_2 + \tilde C_2^*U, \qquad \Delta_2 = I - UU^*.
\end{equation}
We first use $Q_E$ and $Q_U$ to study the reduced system's rank and Hankel approximation error. Then, we prove that our construction of $U$ (see~\cref{sec:stabilizetheta}) makes both $Q_E$ and $Q_U$ small. In addition, we define
\[
	\delta = \min \{\tilde{\sigma}_{j_1-1}  - \tilde{\sigma}_{k+1}, \tilde{\sigma}_{k+1} - \tilde{\sigma}_{j_2+1}\},
\]
where we take $\delta = \tilde{\sigma}_{j_1-1}  - \tilde{\sigma}_{k+1}$ if $j_2 = K$. Of course, we have $\delta > \epsilon$ by definition. Moreover, in most cases where we need a positive $\epsilon$ to select singular values that are numerically close to $\tilde{\sigma}_{k+1}$, we expect that $\delta \gg \epsilon$ because $\tilde{\sigma}_{j_1}, \ldots, \tilde{\sigma}_{j_2}$ form a cluster of singular values that is relatively far away from the rest.

\subsection{Rank of the reduced system}

According to the discussion in~\cref{sec:stabilize}, the rank of the reduced system is equivalent to the number of eigenvalues of $\hat{A}$ in the open left half-plane. Using this fact, we study the rank of the reduced system by proving that under mild conditions, $\hat{A}$ has exactly $k$ eigenvalues in the left half-plane and $K-r-k$ in the right half-plane. In the following statement, in which $\sigma_\rho(X)$ denote the $\rho$-pseudospectrum of the matrix $X$~\cite{trefethen1999spectra}, i.e., $\sigma_\rho(X) = \big\{z \in \C \;{\big|}\; \big\| (zI - X)^{-1} \big\| > \rho^{-1} \big\}$.

\begin{theorem}\label{thm.rankGhat}
Using the notation introduced in~\cref{sec:algorithm} and~\cref{sec:analysis}, suppose any one of the following conditions holds:
\begin{enumerate}[label=(\alph*)]
    \item\label{item.a} $\{z \mid \text{Re}(z) = 0\} \cap \sigma_{\rho_1}(\hat{A}) = \emptyset$, where
    \[
        \rho_1 = \big\|{\tilde C_1}\big\|^2 \delta^{-1} Q_U / 2.
    \]
    \item\label{item.b} $\{z \mid \text{Re}(z) = 0\} \cap \sigma_{\rho_2}(\tilde{A}_{11}) = \emptyset$, where
    \[
        \rho_2 = \rho_1 + \big\|{\tilde B_1}\big\|\sqrt{\big\|{\tilde C_1}\big\|^2 Q_U \delta^{-2} + 2\rho_1 \delta^{-1}}.
    \]
    \item\label{item.c} $\{z \mid \text{Re}(z) = 0\} \cap \sigma_{\rho_3}(\tilde{A}) = \emptyset$, where
    \begin{align*}
        \rho_3 &= 2\rho_2 + \delta^{-1} \big(\epsilon\big\|{\tilde A_{21}}\big\| + \epsilon \big\|{\tilde A_{12}}\big\| + Q_E\sqrt{1\!+\!Q_U}\big\|{\tilde C_1}\big\| + \big\|{\tilde C_1}\big\|\big\|{\tilde C_2}\big\|Q_U \big).
    \end{align*}
\end{enumerate}
Then, $\hat{A}$ has exactly $k$ eigenvalues in the left half-plane and $K-r-k$ eigenvalues in the right half-plane. In particular, the rank of the Hankel operator defined by $(\hat{A}_s,\hat{B}_s,\hat{C}_s,\hat{D}_s)$ is equal to $k$.
\end{theorem}

\begin{proof}
See~\Cref{sec:proof1}.
\end{proof}

One can understand~\Cref{thm.rankGhat} from the Bauer--Fike Theorem~\cite[Thm.~2.3]{trefethen1999spectra}, which states that for a diagonalizable matrix $X = V \Lambda V^{-1}$, we have
\[
	\sigma(X) + B_{\rho}(0) \subset \sigma_{\rho}(X) \subset \sigma(X) + B_{\rho\kappa(V)}(0),
\]
where $\sigma(X)$ is the spectrum of $X$, $B_R(0)$ is the open ball of radius $R$ centered at the origin and $\kappa(V)$ is the condition number of $V$. Hence, the condition that $\{z \mid \text{Re}(z) = 0\} \cap \sigma_{\rho_1}(\hat{A}) = \emptyset$ is saying that if the eigenvector matrix of $\hat{A}$ is well-conditioned, then the eigenvalues of $\hat{A}$ are separated from the imaginary axis by a distance on the order of $\rho_1$. This aligns with our intuition because when we perturbed the inputs of our algorithm to introduce an $\epsilon > 0$, the eigenvalues of $\hat{A}$ are perturbed. If an eigenvalue is too close to the imaginary axis, then there is a chance that it could be perturbed across the imaginary axis, changing the rank of the associated Hankel operator. Therefore, in theory, we can guarantee that the rank of the reduced system is $k$ by taking $\epsilon$ small enough. While~\Cref{thm.rankGhat} gives explicit criteria for guaranteeing that the reduced system has the desired rank, it does not tell us what the perturbed rank is if these criteria are violated. To come up with such a statement, one has to study the perturbation of every eigenvalue of $\hat{A}$. In practice, one can instead use the stabilization procedure in~\cref{sec:stabilize} to reveal the rank of the reduced system automatically.

\subsection{Hankel error of the reduced system}

In the previous subsection, we studied the effect of $\epsilon > 0$ on the rank of the reduced system. To justify our modified HNA algorithm, we also need to understand the approximation error $\big\|{\hat{G} - \tilde{G}}\big\|_H$ under our modification, which is analyzed by the following theorem.

\begin{theorem}\label{thm.error}
    There exists a degree-$4$ multivariate polynomial $R: \R^6 \rightarrow \R$ with positive coefficients, such that given any $\rho > 0$ with $\{z \mid \text{Re}(z) = 0\} \cap \sigma_{\rho}(A) = \emptyset$, the transfer function $\tilde{G}$ associated with $(\tilde A,\tilde B,\tilde C,\tilde D)$ and $\hat{G}(z)$ associated with $(\hat{A},\hat{B},\hat{C},\hat{D})$ satisfy
    \begin{equation}
        \big\|{\hat{G} \!-\! \tilde{G}}\big\|_{\infty}^2 \!\leq\! \tilde{\sigma}_{k+1}^2 \!+\! M\! 
        \left(\!\rho^{-1}\!\!\!\sup_{\text{Re}(z) = 0}\!\!\! \norm{X(z)}^{-1}  \!\!(\epsilon \!+\! Q_E \!+\! Q_U) \!+\! \!\sup_{\text{Re}(z) = 0}\!\!\!\norm{X(z)}^{-2} \!Q_U \!+\! \rho^{-2} \epsilon\!\right),
    \end{equation}
    where $M = R\big(\big\|{\tilde{A}}\big\|, \big\|{\tilde{B}}\big\|, \big\|{\tilde{C}}\big\|$, $\tilde{\sigma}_1, 1+Q_E, 1+Q_U\big)$ and
    \begin{equation*}
        X(z) = -z (\tilde{\Sigma}_1^2 - \tilde{\sigma}_{k+1}^2 I_{K-r}) - \tilde{\sigma}_{k+1}^2 \tilde A_{11} + \tilde{\Sigma}_1 \tilde A_{11}^* \tilde{\Sigma}_1 - \tilde{\sigma}_{k+1} \tilde B_1 U^* \tilde C_1.
    \end{equation*}
\end{theorem}

\begin{proof}
See~\Cref{sec:proof2}.
\end{proof}

One can write an explicit but complicated expression for $R$ (see~\Cref{sec:proof2}). Since $R$ is a fixed polynomial, for a system $(\tilde A,\tilde B,\tilde C,\tilde D)$, we find that $M$ is bounded as $\epsilon \rightarrow 0$ as long as $Q_E$ and $Q_U$ are bounded (see~\Cref{thm.QEQU}). The matrix $X(z)$ adds complication to the expressions. Unfortunately, this term is due to the resolvent $(zI - \tilde{A})^{-1}$ in the transfer function $\tilde{G}$ and cannot be avoided. While it is possible to bound $\norm{X(z)}^{-1}$ using $\delta^{-1}$, in practice, the effect of $\delta^{-1}$ is rarely witnessed (see~\cref{sec:expstability}). In~\cref{sec:blackbox}, we show that given our construction of $U$ (see~\cref{sec:stabilizetheta}), $Q_E$ and $Q_U$ are on the order of $\epsilon$. Hence, unless $\sup_{\text{Re}(z) = 0} \norm{X(z)}^{-1}$ grows rapidly as $\epsilon \rightarrow 0$, $\big\| \hat{G} - G \big\|_H - \tilde{\sigma}_{k+1}$ is on the order of $\epsilon$. This shows that with a small $\epsilon$, the Hankel error of the reduced system produced by the modified HNA algorithm does not deviate from the optimal error, i.e., $\tilde{\sigma}_{k+1}$, by too much.

\subsection{Completing the analysis of the modified HNA algorithm}\label{sec:blackbox}
In the previous subsections, we showed the consequences of modifying the HNA algorithm on the rank and the approximation error of the reduced system. The claim that the modified HNA algorithm is near-optimal is based on the assumption that $Q_E$ and $Q_U$ are small, where $Q_E$ and $Q_U$ depend on $U$. Here, we provide upper bounds on $Q_E$ and $Q_U$.

\begin{theorem}\label{thm.QEQU}
Using the notation introduced in~\cref{sec:algorithm} and~\cref{sec:analysis}, let $s_1 \geq \cdots \geq s_p \geq 0$ be the singular values of $\tilde{C}_2$ and $s_{q} - s_{q+1}$, where we define $s_{p+1} = 0$ when $p = q$. The following statements hold:
    \begin{enumerate}[label=(\alph*)]
        \item \label{item.lem1a} The quantity $Q_E$ satisfies
        \begin{equation}
            Q_E \leq Q_1 + Q_2,
        \end{equation}
        where
        \[
            Q_1 = \gamma + \sqrt{\gamma^2 + 4\epsilon\big\|{\tilde A_{22}}\big\|}, \qquad  Q_2 = \frac{4\sqrt{2}\epsilon\big\|{\tilde A_{22}}\big\|_F}{(s_{q} - s_{q+1}) - 4\epsilon\big\|{\tilde A_{22}}\big\|_F} \big\| \tilde{B}_2 \big\|.
        \]
        Moreover, $Q_1 = 0$ when $q = r$ and $Q_2 = 0$ when $q = r$ or $\tilde A_{22}$ is Hermitian.
        \item \label{item.lem1b} The quantity $Q_U$ satisfies
        \begin{equation}
            Q_U \leq 4s_q^{-2} \epsilon \big\|{\tilde A_{22}}\big\|,
        \end{equation}
        which is zero when $\tilde A_{22}$ is Hermitian.
    \end{enumerate}
\end{theorem}
\begin{proof}
See~\Cref{sec:proof3}.
\end{proof}

In most circumstances, $r$ is a very small integer since it is unlikely that a lot of singular values of $\Gamma_{\tilde{G}}$ are numerically close to each other (relative to their magnitudes), in which case one only needs to be concerned with $Q_U$ by taking $q = r$. When $r > p$ or $\tilde C_2^*$ contains very small singular values, we must take $q < r$. Then, we have a problem balancing $\gamma$ and $s_q^{-1}$. That is, as $q$ increases, $\gamma$ decreases but $s_q^{-1}$ increases, so we need to select $q$ so that $Q_E \approx Q_U$. The factor $\nu$ that appears in $Q_2$ comes from the theory of singular vectors perturbation. It is inevitable in the worst-case scenario, although the upper bound is rarely attained in practice.

\subsection{An experiment that corroborates the theory of the modified HNA algorithm}\label{sec:expstability}

In~\Cref{thm.error}, we provide an upper bound on the error of $\big\|{{\tilde{G}} - {\hat{G}}}\big\|_{\infty}$. The error mainly depends on the value of $\epsilon$, and we observe that $\delta$ does not play an important role in the approximation error. To numerically verify~\Cref{thm.error}, we let $A \in \C^{16 \times 16}$ be a random Hermitian, negative definite matrix and 
\[
P \!=\! Q \!=\! \diag\!\left(\!\frac{1}{10}, \ldots, \frac{4}{10}, \frac{5}{10}\!-\!\delta, \frac{5}{10}\!-\!\epsilon, \frac{5}{10}\!-\!\frac{\epsilon}{2}, \frac{5}{10}, \frac{5}{10}\!+\!\frac{\epsilon}{2}, \frac{5}{10}\!+\!\epsilon, \frac{5}{10}\!+\!\delta, \frac{6}{10}, \ldots, 1\!\right)\!.
\]
We select $\epsilon$ and $\delta$ from a logarithmic grid. For each pair of $\epsilon$ and $\delta$, since $A$ is Hermitian and negative definite, we can compute $B$ and $C$ that satisfy the Lyapunov's equations~\cref{eq.lyapunov} by computing the Cholesky factorization of $-AP-PA^*$ and $-A^*Q-QA$, respectively. We then apply the modified HNA algorithm with ${\sigma}_{k+1} = 1/2$ and $(\sigma_{j_1}, \ldots, \sigma_{j_2}) = (1/2-\epsilon,1/2-\epsilon/2,1/2,1/2+\epsilon/2, 1/2+\epsilon)$. Let $G$ be the transfer function of the system $(A,B,C,D)$. In~\Cref{fig:stability}a, we observe that $\big\|{G - \hat{G}}\big\|_\infty - 1/2 = \mathcal{O}(\epsilon)$, whereas it does not depend on $\delta$ in this example. Surprisingly,~\Cref{fig:stability}a shows that the algorithm works even when $\delta < \epsilon$, i.e., when we separate out $\sigma_{5},\sigma_{6},\sigma_{8},\sigma_{10}$ and $\sigma_{11}$ but leave $\sigma_{7}$ and $\sigma_{9}$ in ${\Sigma}_1$. When $\epsilon$ and $\delta$ are both small, we find a region in~\Cref{fig:stability}a where the errors do not follow the general pattern. This happens because the matrices $\hat{A}$ and $\hat{B}$ contain large entries, making computation with the reduced matrices (i.e., evaluating the transfer function $\hat{G}$) numerically unstable (see~\Cref{fig:stability}b). This experimentally justifies the necessity of the modified HNA algorithm.

\begin{figure}
\centering
\begin{minipage}{.48\textwidth}
\begin{overpic}[width=1\textwidth]{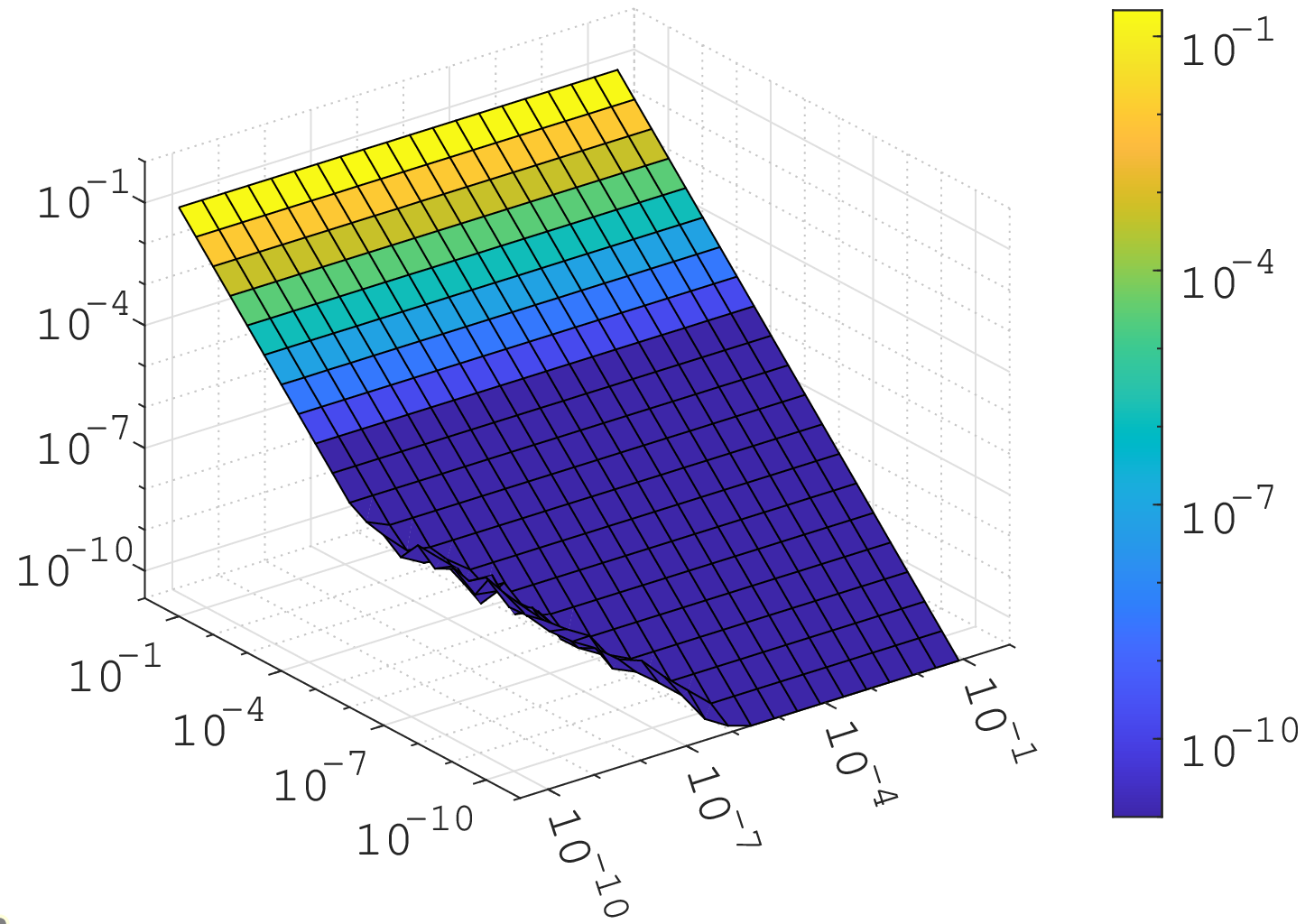}
\put(63,2) {\footnotesize $\delta$}
\put(15,7) {\footnotesize $\epsilon$}
\put(-6,26) {\rotatebox{90}{\footnotesize $\big\|{G \!-\! \hat{G}}\big\|_\infty \!\!\!-\! 0.5$}}
\end{overpic}
\caption*{\footnotesize (a) Approximation error}
\end{minipage}
\begin{minipage}{.48\textwidth}
\begin{overpic}[width=1\textwidth]{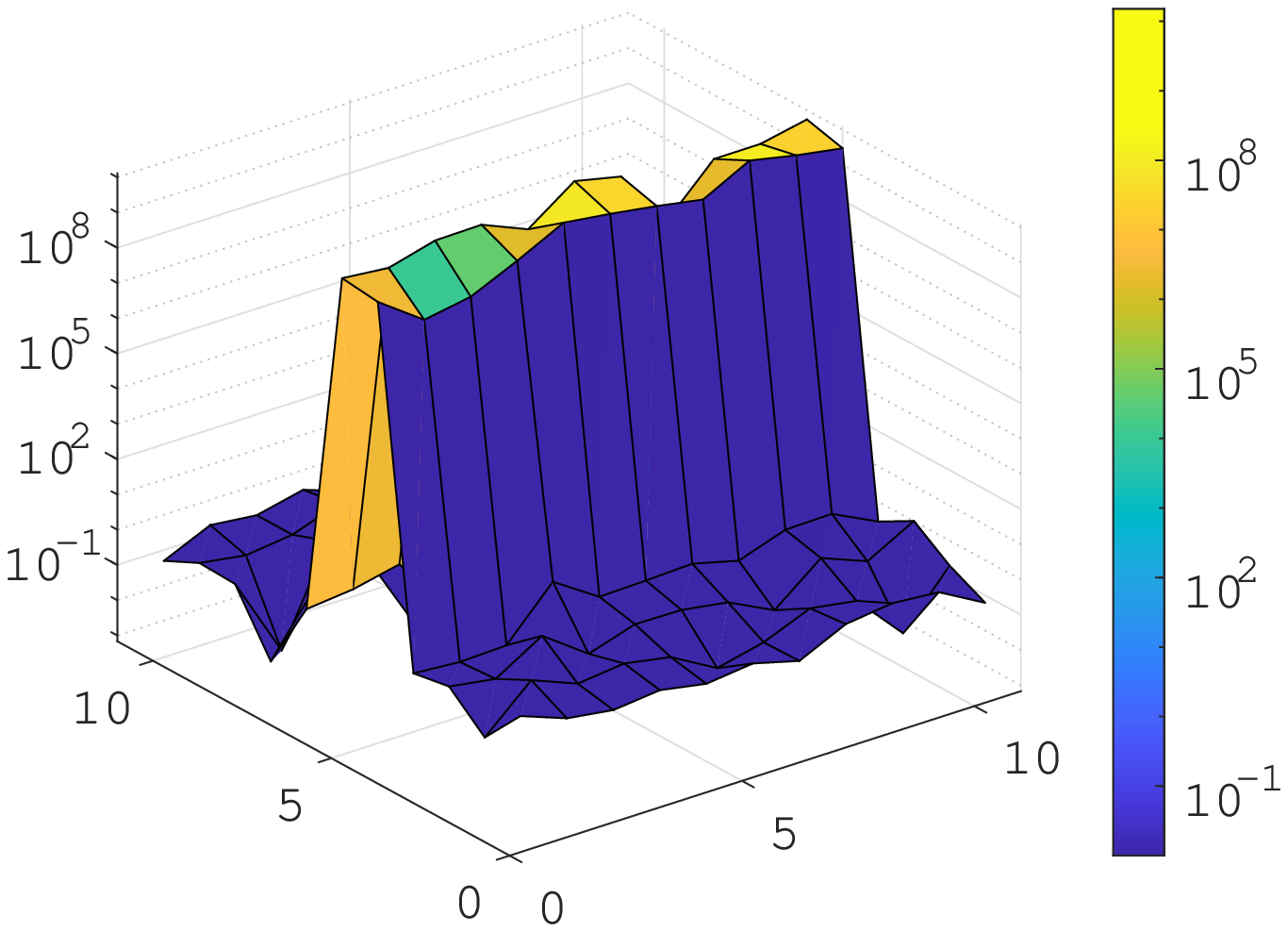}
\put(62,4) {\footnotesize $i$}
\put(18,9) {\footnotesize $j$}
\put(-3,35) {\rotatebox{90}{\footnotesize ${\big|}{\hat{A}_{i,j}}{\big|}$}}
\end{overpic}
\caption*{\footnotesize (b) Sizes of matrix entries}
\end{minipage}
\caption{Left: The approximation error $\big\|{G - \hat{G}}\big\|_\infty - 0.5$ as a function of $\epsilon$ and $\delta$. Right: The absolute values of matrix entires of $\hat{A}$ when $\epsilon = 1 \times 10^{-11}$ and $\delta = 2.5 \times 10^{-11}$.}
\label{fig:stability}
\end{figure}

\section{Adaptations of our two-stage algorithm}\label{sec:adaptation}

Our two-stage algorithm (see \Cref{alg:hankelAAA}) solves Problem~\ref{prob.MRHankel}, which approximates the system by a reduced one in the Hankel norm. Sometimes, however, we must consider a stronger approximation metric that guarantees the reduced system inherits more features from the original one. It turns out that one can do so by using the $\norm{\cdot}_\infty$ norm (see~\cref{eq.Hinfnorm}):

\begin{problem}\label{prob.MRHinf}
Find a stable reduced system $(\hat{A},\hat{B},\hat{C},\hat{D})$ with the transfer function $\hat{G}$, where $\hat{A} \in \C^{k \times k}, \hat{B} \in \C^{k \times m}, \hat{C} \in \C^{p \times k}$, and $\hat{D} \in \C^{p \times m}$, that minimizes $\big\| G - \hat{G} \big\|_{\infty}$.
\end{problem}

By~\cref{eq.transferlaplace}, if the transfer functions of the two systems are close pointwise on the imaginary axis, then the outputs of the two systems given the same input are close. Therefore, Problem~\ref{prob.MRHinf} is ideal for preserving the frequency response of the system. The implementation discussed in~\cref{sec:algorithm}, where we ``throw away" the $H^\infty$ functions when stabilizing the systems (see~\cref{sec:stabilize}), does not suffice to solve Problem~\ref{prob.MRHinf}. To modify the implementation of line~\ref{state.endAAA} and~\ref{state.endHNA} in~\Cref{alg:hankelAAA}, we instead apply the procedure in~\cite[sect.~3.3]{kohler2014closest}, which is implemented and tested in~\cite{gosea2021enforcing}. This procedure approximates an unstable system by a stable system of the same rank and guarantees that their transfer functions are the closest in the infinity norm. Instead of using this stabilizing algorithm for a generic unstable system, one can use an \textit{ad hoc} post-processing procedure discussed in Glover's paper~\cite[sect.~10]{glover1984optimal}, which was designed for the HNA outputs, to stabilize the final reduced system (i.e., line~\ref{state.endHNA} of~\Cref{alg:hankelAAA}). This algorithm has an explicit error bound on the $\norm{\cdot}_\infty$ error of the stabilized transfer function (see~\cite[Thm.~9.7]{glover1984optimal}). 

\section{Selecting the parameters in our two-stage algorithm}\label{sec:parameters}

\Cref{alg:hankelAAA} involves several input values we must select. The parameters $\epsilon$ and $\gamma$ are used to avoid numerical issues, and in practice, we often want to set them to numbers smaller than $\tilde{\sigma}_{k+1}$ and larger than machine precision.

In some cases, we are given a set of samples and do not have control of the locations of $z_i$'s. However, when we can select the samples' locations, our choice impacts the block-AAA algorithm's performance. The imaginary axis is an unbounded domain, and to sample this domain, we take the M\"obius transforms given by 
\[
	z = \frac{s-1}{s+1} \quad \Leftrightarrow \quad s = \frac{1+z}{1-z},
\]
which map between the imaginary axis and the unit circle.\footnote{In the setting for discrete-time LTI systems, the transfer functions are typically analyzed on the unit circle~\cite{glover1984optimal}.} Without any prior information about the transfer function, a natural way of choosing the samples is to place an evenly spaced grid of size $N$ on the unit circle and then map them onto the imaginary axis using a M\"obius transform. When specific properties of the transfer function $G$ are known, samples can be placed in a more informed way, e.g., we can make the samples cluster in regions where $G$ is more oscillatory (see~\cref{sec:expcdplayer}).

The choice of $d$ is also crucial because it controls the time-accuracy trade-off of our two-stage algorithm. If we assume that $m$ and $p$ are constants, then constructing the rational approximation using the block-AAA algorithm takes $\mathcal{O}(Nd^3)$ flops~\cite{nakatsukasa2018AAA}. Since $K = \mathcal{O}(d)$, applying the modified HNA algorithm takes $\mathcal{O}(d^3)$ flops. Overall, the time complexity of our two-stage algorithm is $\mathcal{O}(Nd^3)$.  We want $d$ to be large enough so that the first stage produces an intermediate system that is accurate and yet not too large to make our algorithm computationally inefficient. In practice, we can choose $d$ adaptively in the block-AAA algorithm. That is, we predetermine an error threshold and a maximum degree $d_{\max}$. We run the block-AAA algorithm until the error gets lower than the threshold or the degree reaches $d_{\max}$.

The regularization parameter $\lambda$ in~\cref{eq.regularizedAAA} is hard to determine a priori. However, one can choose $\lambda$ by running the block-AAA algorithm and comparing two errors. Let $\text{fl}(\tilde{A})$, $\text{fl}(\tilde{B})$, and $\text{fl}(\tilde{C})$ be the floating-point representations of the matrices $\tilde{A}, \tilde{B}$, and $\tilde{C}$ (see~\cref{sec:stableAAA}). Let $\text{fl}(R_d)(z)$ be the evaluation of $\text{fl}(\tilde{C}) (zI - \text{fl}(\tilde{A}))^{-1} \text{fl}(\tilde{B})$ in floating-point arithmetic. We consider the block-AAA error $E_1 = \norm{G - R_d}_\infty$ and the numerical error $E_2 = \norm{R_d - \text{fl}(R_d)}_\infty$. As $\lambda$ increases, we expect that $E_1$ increases because the least-squares system~\cref{eq.lsqblockAAA} is solved less accurately and that $E_2$ decreases because $\norm{W^{(d)}}_F$ is suppressed by the regularization, which cures the numerical issues discussed in~\cref{sec:stableAAA}. We leverage the regularized least-squares solver when we make $E_1$ and $E_2$ to have the same order of magnitude. In practice, using the bisection method, one can search for such a $\lambda$.

\section{Experiments on real-world problems}\label{sec:expdis}

This section presents two experiments to demonstrate our two-stage algorithm on real-world problems. 

\subsection{An experiment on the ATMOS example}\label{sec:expeady}

The first example models the track of a storm in the atmosphere~\cite{chahlaoui2005benchmark,farrell1995stochastic}. In this example, $A$ is a $598 \times 598$ matrix and $m = p = 1$. The Hankel singular values decay relatively fast. We apply our two-stage algorithm with different values of $d$ to compute a rank-$10$ reduced model.  We sample from evenly spaced points on the unit circle (see~\cref{sec:parameters}). We apply three algorithms: balanced truncation (BT), our two-stage algorithm that is based on the least-squares solver in the AAA algorithm (LS), and one that is based on the regularized least-squares solver (RLS). In~\Cref{fig:eady}a, we see that as $d$ increases, using the least-squares solver induces numerical instabilities and causes a decrease in the accuracy, while using the regularized least-squares solver rectifies the issue. In~\Cref{fig:eady}b, we see that our two-stage algorithm runs faster than a pure SVD-based algorithm, and yet the approximation error is tiny.

\begin{figure}
\centering

\begin{minipage}{.48\textwidth}
\begin{overpic}[width=1\textwidth]{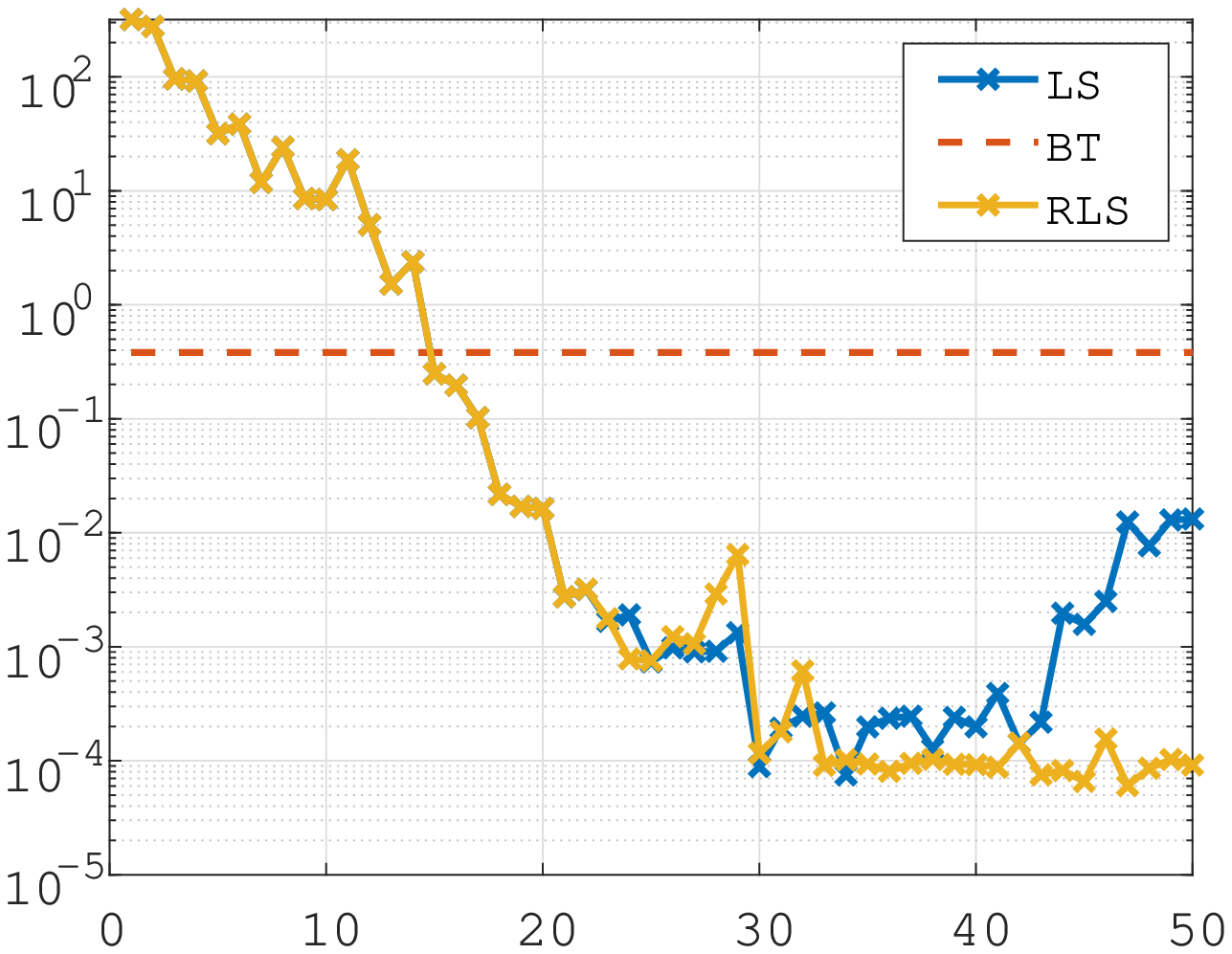}
\put(50,-2) {\footnotesize $d$}
\put(-2,20) {\rotatebox{90}{\footnotesize $\|{\hat{G} - G}\|_H - \sigma_{11}$}}
\end{overpic}
\caption*{\footnotesize (a) Approximation error}
\end{minipage}
\begin{minipage}{.48\textwidth}
\begin{overpic}[width=1\textwidth]{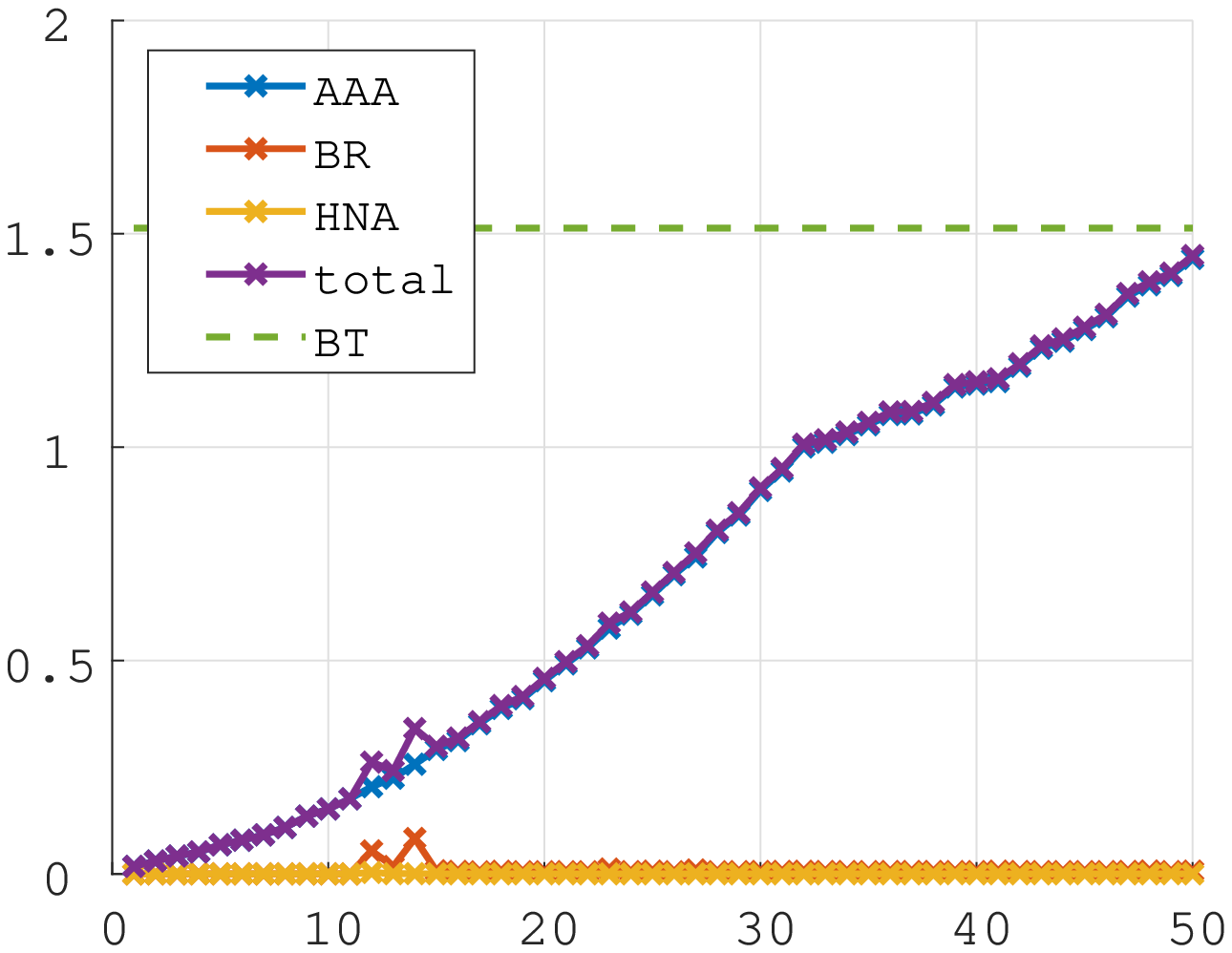}
\put(52,-2) {\footnotesize $d$}
\put(2,37) {\rotatebox{90}{\footnotesize time}}
\end{overpic}
\caption*{\footnotesize (b) Execution time}
\end{minipage}
\caption{Model reduction on the ATMOS example. Left: The approximation error minus the optimal error of our two-stage algorithm as $d$ increases and the approximation error of the balanced truncation output. Right: The execution time of the block-AAA, balanced realization (BR), and modified HNA algorithms as $d$ increases, and the time of the balanced truncation algorithm.}
\label{fig:eady}
\end{figure}

\subsection{An experiment on the C-DISC example}\label{sec:expcdplayer}
Next, we consider an example of track-following for a CD player. Details of the problem specification and the model can be found in~\cite{chahlaoui2005benchmark,draijer1992adaptive,wortelboer1996closed}. In this example, we have $A \in \C^{120 \times 120}$ and $m = p = 2$. This is a multiple-input multiple-output (MIMO) model, and we seek a rank-$10$ Hankel approximation. 
Since the transfer function changes more rapidly near $z = 0$, we sample logarithmically on $[-10^3,-10]i$ and $[10,10^3]i$ on the imaginary axis (see~\Cref{fig:CDplayer}a). This problem is multi-scale in the sense that the magnitude of the transfer function is at the order of $10^5$, whereas $\sigma_{11} = 8.70$. This makes it challenging to construct the intermediate system in a numerically stable way (see~\cref{sec:stableAAA}). To attack this issue, we apply the block-AAA algorithm twice with different values of $\lambda$. First, we construct a degree-$10$ rational approximation $R_{10}^{(1)}$ of $G$ using $\lambda = 10^{-2}$ to obtain a coarse approximation that reduces the scale of $G - R_{10}^{(1)}$. Then, we set $\lambda = 10^{-9}$ to obtain an accurate approximation $R_{d-10}^{(2)}$ of $G - R_{10}^{(1)}$. The systems associated with $R_{10}^{(1)}$ and $R_{d-10}^{(2)}$ can then be concatenated to obtain the intermediate system. In~\Cref{fig:CDplayer}b, we see that applying the block-AAA algorithm twice outperforms using a single execution of the block-AAA algorithm. In both cases, using the regularized least-squares solver is important for the numerical stability concerns.

\begin{figure}
\centering
	\begin{minipage}{.48\textwidth}
	\begin{overpic}[width=1\textwidth]{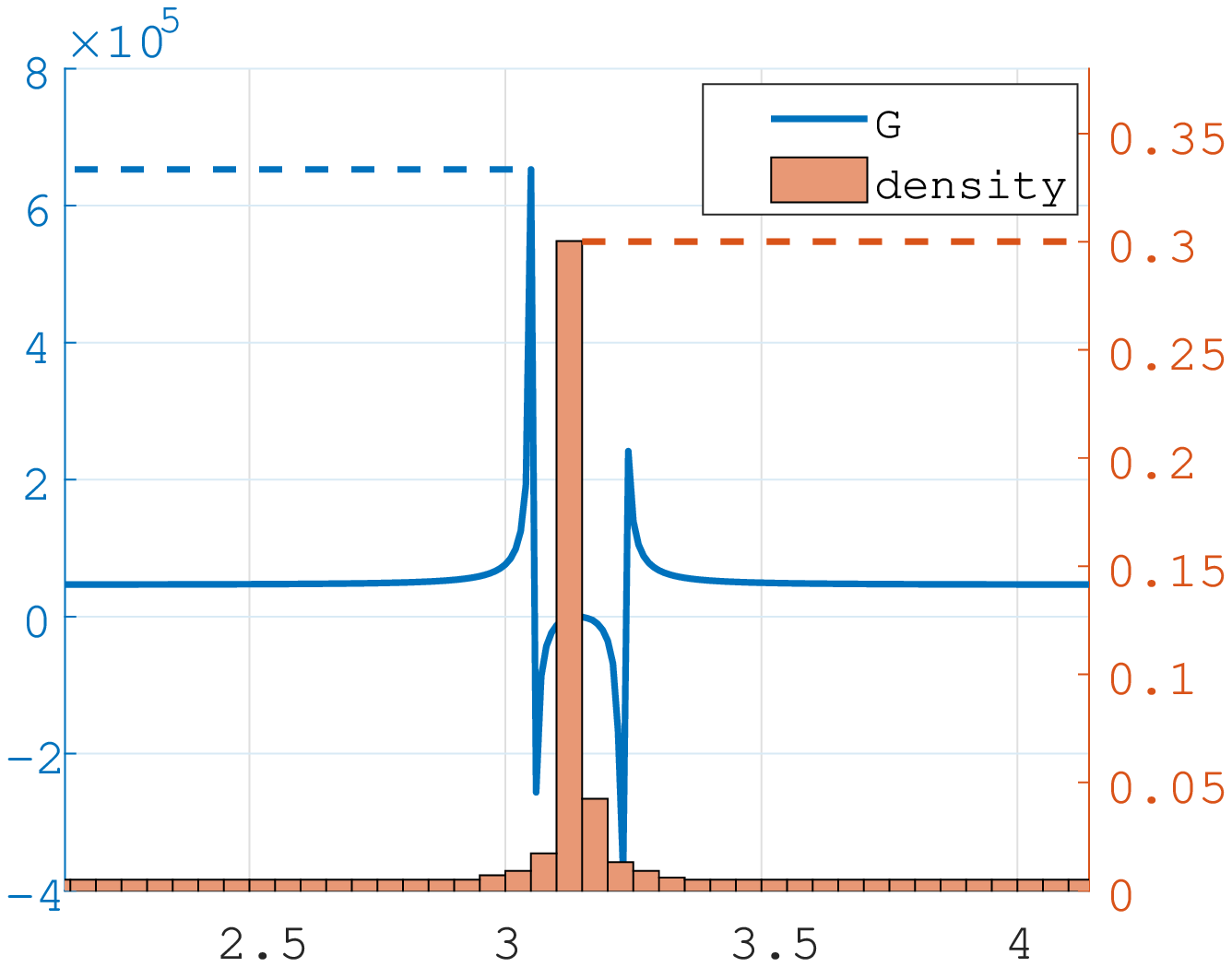}
	\put(50,-2) {\footnotesize $\theta$}
	\put(4,20) {\rotatebox{90}{\footnotesize transfer function}}
	\put(100,45) {\rotatebox{270}{\footnotesize density}}
	\end{overpic}
	\caption*{\footnotesize (a) Transfer function and samples' distribution}
	\end{minipage}
	\quad
	\begin{minipage}{.48\textwidth}
	\begin{overpic}[width=1\textwidth]{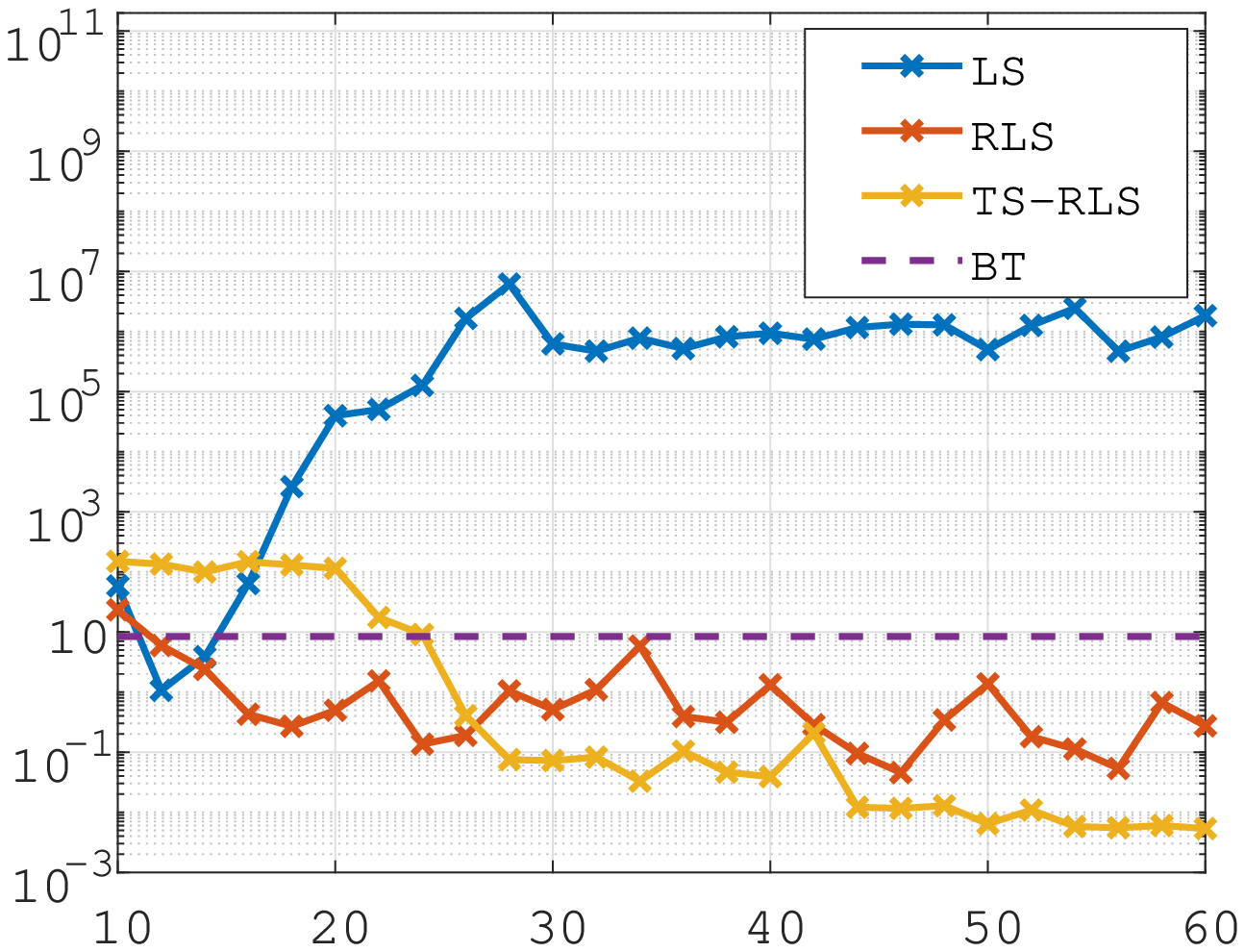}
	\put(50,-2) {\footnotesize $d$}
	\put(0,20) {\rotatebox{90}{\footnotesize $\|{\hat{G} - G}\|_H - \sigma_{11}$}}
	\end{overpic}
	\caption*{\footnotesize (b) Approximation error}
	\end{minipage}

\caption{Model reduction on the C-DISC example. Left: The transfer function $G$ (left vertical axis) and the density of the sample locations (right vertical axis) on the unit circle (see~\cref{sec:parameters}) parameterized by the angle $\theta$. Right: The approximation error minus the optimal error of our two-stage algorithm as $d$ increases and the approximation error of the balanced truncation output.}
\label{fig:CDplayer}
\end{figure}

\section{Conclusion}\label{sec:conclusion}

This paper proposed a robust and efficient two-stage algorithm for reducing an LTI system based on the block-AAA and modified HNA algorithms. Our two-stage algorithm has a parameter $d$ that allows us to choose a balance between computational time and accuracy. The algorithm contained many parameters and interchangeable components, which allows an independent study of any combination of them to develop an accurate, fast, and numerically stable algorithm for model reduction.

\appendix

\section{Proof of~\Cref{thm.rankGhat}}\label{sec:proof1}

This section presents the proof of~\Cref{thm.rankGhat}. The matrix $\tilde{\Sigma}_1^2 -\hat{\sigma}^2 I$ plays an important role in the analysis, and we denote it by $\Phi$. Recall that $\tilde{\Sigma}_2 \!\!=\! \text{diag}(\tilde \sigma_{j_1}, \ldots, \tilde \sigma_{j_2})$. We also introduce a list of auxiliary matrices appearing in the technical equations in~\Cref{lem.technicals}. We define an augmented system by
\begin{equation}
    A_e \!=\! \left[
    \begin{array}{c|c|c}
        \tilde A_{11} & \tilde A_{12} & 0 \\
        \hline
        \tilde A_{21} & \tilde A_{22} & 0 \\
        \hline
        0 & 0 & \hat{A}
    \end{array}
    \right], \quad
    B_e \!=\! \left[
    \begin{array}{c}
         \tilde B_1 \\
         \hline
         \tilde B_2 \\
         \hline
         \hat{B}
    \end{array}
    \right], \quad
    C_e \!=\! \left[
    \begin{array}{c|c|c}
         \tilde C_1 & \tilde C_2 & -\hat{C}
    \end{array}
    \right], \quad
    D_e \!=\! \tilde D - \hat{D},
\end{equation}
and let
\begin{equation}
    P_e = 
    \left[
    \begin{array}{c|c|c}
        \tilde{\Sigma}_1 & 0 & I_{K-r} \\
        \hline
        0 & \hat{\sigma}I_r & 0 \\
        \hline
        I_{K-r} & 0 & \tilde{\Sigma}_1\Phi^{-1}
    \end{array}
    \right].
\end{equation}
When $\epsilon = 0$, $P_e$ is the controllability gramian of the system $(A_e,B_e,C_e,D_e)$, i.e., it satisfies $A_eP_e + P_eA_e^* + B_eB_e^* = 0$~\cite{glover1984optimal}. However, the equation is not satisfied with a positive $\epsilon > 0$. We define the error term as
\begin{equation}\label{eq.perturbedlyapunov}
\begin{aligned}
&A_eP_e \!+\! P_eA_e^* \!+\! B_eB_e^* = \Omega = \left[
\begin{array}{c|c|c}
     \Omega_{11} & \Omega_{12} & \Omega_{13}  \\
     \hline
     \Omega_{12}^* & \Omega_{22} & \Omega_{23}  \\
     \hline
     \Omega_{13}^* & \Omega_{23}^* & \Omega_{33}  \\
\end{array}
    \right] \\
        &=\!\! \left[
            \begin{array}{c|c|c}
                \!\!\!\!\tilde A_{11}\!\tilde{\Sigma}_1 \!+\! \tilde{\Sigma}_1\! \tilde A_{11}^* \!+\! \tilde B_1 \tilde B_1^*\!\!\! &\!\! \hat{\sigma}\tilde A_{12} \!+\! \tilde{\Sigma}_1 \tilde A_{21}^* \!+\! \tilde B_1 \tilde B_2^* & \tilde A_{11} + \hat{A}^* \!+\! \tilde B_1\hat{B}^* \\
                \hline
                \!\!\!\hat{\sigma} \tilde A_{12}^* \!+\! \tilde A_{21}\tilde{\Sigma}_1 \!+\! \tilde B_2 \tilde B_1^*\!\!\! &\!\! \hat{\sigma}\tilde A_{22} \!+\! \hat{\sigma} \tilde A_{22}^* \!+\! \tilde B_2 \tilde B_2^* & \tilde A_{21} + \tilde B_2 \hat{B}^* \\
                \hline
               \tilde A_{11}^* \!+\! \hat{A} \!+\! \hat{B} \tilde B_1^* &\!\! \tilde A_{21}^* \!+\! \hat{B} \tilde B_2^* & \!\!\hat{A}\tilde{\Sigma}_1\Phi^{-1} \!\!+\! \tilde{\Sigma}_1\Phi^{-1}\!\! \hat{A}^* \!\!+\! \hat{B}\hat{B}^*\!\!\!\!
            \end{array}
        \right]\!.
        \end{aligned}
\end{equation}
Let $\Omega_1$ be the top-left $K \times K$ sub-block of $\Omega$, $\Omega_2$ be the top-right $K \times (K-r)$ sub-block of $\Omega$, and $\Omega_3$ be the bottom-right $(K-r) \times (K-r)$ sub-block of $\Omega$. We provide a list of technical equations used in the analysis. Some of them are not used until the next section. However, we present all technical equations here for cleanness.

\begin{lemma}\label{lem.technicals}
The following equations hold:
\begin{align}
   &B_2 U^* =  \Delta_1U^* + C_2^* \Delta_2 - C_2^*. \label{eq.secondU}\\
   &\tilde A_{11} + \hat{A}^* + \tilde B_1\hat{B}^* = 0, \label{eq.technical-1}\\
    &\tilde A_{21} \!+\! \tilde B_2\hat{B}^* \!=\! (\hat{\sigma}(\tilde{\Sigma}_2 \!-\! \hat{\sigma}I)\tilde A_{21} \!+\! (\hat{\sigma}I \!-\! \tilde{\Sigma}_2)\tilde A_{12}^* \tilde{\Sigma}_1 \!+\! \hat{\sigma}\Delta_1U^*\tilde C_1 \!+\! \hat{\sigma} \tilde C_2^* \Delta_2 \tilde C_1)\Phi^{-1}, \label{eq.technical0}\\
    &\hat{A} \tilde{\Sigma}_1 \Phi^{-1} \!+\! \tilde{\Sigma}_1 \Phi^{-1} \hat{A}^* \!+\! \hat{B}\hat{B}^* = -\hat{\sigma}^2 \Phi^{-1} \tilde C_1^* \Delta_2 \tilde C_1\Phi^{-1}, \label{eq.technical1}\\
    &\tilde A_{11}^*\Phi + \Phi \hat{A} + \tilde C_1^* \hat{C} = 0, \label{eq.technical2}\\
    &\tilde A_{12}^* \Phi + \tilde C_2^* \hat{C} = \hat{\sigma} (\tilde{\Sigma}_2 - \hat{\sigma}I) \tilde A_{12}^* - (\tilde{\Sigma}_2 - \hat{\sigma}I) \tilde A_{21} \tilde{\Sigma}_1 + \hat{\sigma}\Delta_1\tilde B_1^*, \label{eq.technical3} \\
    &\tilde{\Delta} := D_e B_e^* + C_eP_e = \left[
        \begin{array}{c|c|c}
            0 & \hat{\sigma} (U\Delta_1^* + \Delta_2^* \tilde C_2) & - \hat{\sigma}^2 \Delta_2 \tilde C_1 \Phi^{-1}
        \end{array}
        \right]. \label{eq.technical4}
\end{align}
\end{lemma}
\begin{proof}
To prove~\cref{eq.secondU}, we have
\begin{equation*}
	B_2 U^* = (\Delta_1 - C_2^* U)U^* = \Delta_1U^* - C_2^* UU^* = \Delta_1U^* + C_2^* \Delta_2 - C_2^*.
\end{equation*}
Next, note that
\begin{align*}
    \tilde B_1\hat{B}^* &= (\tilde B_1 \tilde B_1^*\tilde{\Sigma}_1 + \hat{\sigma} \tilde B_1 U^* \tilde C_1)\Phi^{-1} \\
    &= (-\tilde A_{11}\tilde{\Sigma}_1^2 - \tilde{\Sigma}_1\tilde A_{11}^*\tilde{\Sigma}_1 + \hat{\sigma} \tilde B_1 U^* \tilde C_1)\Phi^{-1} = -\tilde A_{11} - \hat{A}^*,
\end{align*}
where the first step follows from the Lyapunov equations~\cref{eq.lyapunov} and the second step follows from the definition of $\hat{A}$. This proves~\cref{eq.technical-1}. Next, we have
\begin{align*}
    &\tilde B_2 \hat{B}^* =\! \tilde B_2(\tilde B_1^*\tilde{\Sigma}_1 + \hat{\sigma}U^* \tilde C_1)\Phi^{-1} = (\tilde B_2 \tilde B_1^*\tilde{\Sigma}_1 - \hat{\sigma} \tilde C_2^* \tilde C_1 + \hat{\sigma}\Delta_1U^* \tilde C_1 + \hat{\sigma} \tilde C_2^* \Delta_2 \tilde C_1)\Phi^{-1} \\
    &\!= \!([-\tilde A_{21}\tilde{\Sigma}_1 - \tilde{\Sigma}_2 \tilde A_{12}^*]\tilde{\Sigma}_1 + \hat{\sigma}[\tilde A_{12}^*\tilde{\Sigma}_1 + \tilde{\Sigma}_2 \tilde A_{21}] + \hat{\sigma}\Delta_1U^* \tilde C_1 + \hat{\sigma} \tilde C_2^* \Delta_2 \tilde C_1)\Phi^{-1} \\
    &\!= \!(\!(\!-\hat{\sigma}^2 \tilde A_{21} \!\!-\!\! \tilde{\Sigma}_2 \tilde A_{12}^*\!)\tilde{\Sigma}_1 \!\!+\! \hat{\sigma}\!(\!\tilde A_{12}^*\tilde{\Sigma}_1 \!\!+\!\! \tilde{\Sigma}_2 \tilde A_{21}\!) \!\!+\! \hat{\sigma}\Delta_1U^* \tilde C_1 \!\!+\! \hat{\sigma} \tilde C_2^* \Delta_2 \tilde C_1\!)\Phi^{\!-1} \!\!\!-\!\! \tilde A_{21}\!(\tilde{\Sigma}_1^2 \!\!-\! \hat{\sigma}^2I) \Phi^{\!-1} \\
    &\!= \!(\hat{\sigma}(\tilde{\Sigma}_2 - \hat{\sigma}I) \tilde A_{21} + (\hat{\sigma}I - \tilde{\Sigma}_2) \tilde A_{12}^* \tilde{\Sigma}_1 + \hat{\sigma}\Delta_1U^* \tilde C_1 + \hat{\sigma} \tilde C_2^* \Delta_2 \tilde C_1)\Phi^{-1} - \tilde A_{21},
\end{align*}
which proves~\cref{eq.technical0}. To verify~\cref{eq.technical1}, we have
\begin{align*}
    \Phi \hat{A} \tilde{\Sigma}_1 \!+\! \tilde{\Sigma}_1 \hat{A}^* \Phi &\!=\!  (\hat{\sigma}^2 \tilde {A}_{11}^* \!\!+\! \tilde{\Sigma}_1 \tilde {A}_{11} \tilde{\Sigma}_1 \!\!-\! \hat{\sigma} \tilde {C}_1^* U \tilde {B}_1^*) \tilde{\Sigma}_1 \!\!+\! \tilde{\Sigma}_1 (\hat{\sigma}^2 \tilde {A}_{11}^* \!\!+\! \tilde{\Sigma}_1 \tilde {A}_{11} \tilde{\Sigma}_1 \!\!-\! \hat{\sigma} \tilde {C}_1^* U \tilde {B}_1^*)^* \\
    &\!=\! (\hat{\sigma}^2 \tilde {A}_{11}^* \!\!+\! \tilde{\Sigma}_1 \tilde {A}_{11} \tilde{\Sigma}_1 \!\!-\! \hat{\sigma} \tilde {C}_1^* U \tilde {B}_1^*) \tilde{\Sigma}_1 \!\!+\! \tilde{\Sigma}_1 (\hat{\sigma}^2 \tilde {A}_{11} \!\!+\! \tilde{\Sigma}_1 \tilde {A}_{11}^* \tilde{\Sigma}_1 \!\!-\! \hat{\sigma} \tilde {B}_1 U^* \tilde {C}_1) \\
    &\!=\! -(\tilde{\Sigma}_1 \tilde B_1 + \hat{\sigma} \tilde C_1^* U)(\tilde B_1^*\tilde{\Sigma}_1 + \hat{\sigma}U^* \tilde C_1) + \hat{\sigma}^2(\tilde A_{11}^* \tilde{\Sigma}_1 + \tilde{\Sigma}_1 \tilde A_{11} + \tilde C_1^* \tilde C_1) \\
    &\qquad + \tilde{\Sigma}_1 (\tilde A_{11}\tilde{\Sigma}_1 + \tilde{\Sigma}_1 \tilde A_{11}^* + \tilde B_1 \tilde B_1^*) \tilde{\Sigma}_1 + \hat{\sigma}^2 \tilde C_1^* (UU^* - I) \tilde C_1  \\
    &\!=\! -\Phi\hat{B}\hat{B}^*\Phi - \hat{\sigma}^2 \tilde C_1^* \Delta_2 \tilde C_1,
\end{align*}
where the last step follows from the definition of $\hat{B}$ and the Lyapunov equations. Multiplying $\Phi^{-1}$ on the left and the right, we obtain~\cref{eq.technical1}. Next, we have
\begin{align*}
    \tilde C_1^* \hat{C} &= \tilde C_1^*(\tilde C_1 \tilde{\Sigma}_1 + \hat{\sigma} U \tilde B_1^*) = -(\tilde A_{11}^* \tilde{\Sigma}_1 + \tilde{\Sigma}_1 \tilde A_{11}) \tilde{\Sigma}_1 + \hat{\sigma} \tilde C_1^* U \tilde B_1^* \\
    &= -\Phi \hat{A} + \hat{\sigma}^2 \tilde A_{11}^* - \tilde A_{11}^* \tilde{\Sigma}_1^2 = -\Phi \hat{A} - \tilde A_{11}^* \Phi
\end{align*}
and
\begin{align*}
    \tilde C_2^* \hat{C} &= \tilde C_2^*(\tilde C_1 \tilde{\Sigma}_1 + \hat{\sigma} U \tilde B_1^*) = (-\tilde A_{12}^* \tilde{\Sigma}_1 - \tilde{\Sigma}_2 \tilde A_{21})\tilde{\Sigma}_1 - \hat{\sigma} \tilde B_2 \tilde B_1^* + \hat{\sigma}\Delta_1\tilde B_1^* \\
    &= (-\tilde A_{12}^* \tilde{\Sigma}_1 - \tilde{\Sigma}_2 \tilde A_{21})\tilde{\Sigma}_1 + \hat{\sigma} (\tilde A_{21}\tilde{\Sigma}_1 + \tilde{\Sigma}_2 \tilde A_{12}^*) + \hat{\sigma}\Delta_1\tilde B_1^* \\
    &= -\tilde A_{12}^*\Phi - (\tilde{\Sigma}_2 - \hat{\sigma}I) \tilde A_{21} \tilde{\Sigma}_1 + \hat{\sigma} (\tilde{\Sigma}_2 - \hat{\sigma}I) \tilde A_{12}^* + \hat{\sigma}\Delta_1\tilde B_1^*,
\end{align*}
which prove~\cref{eq.technical2} and~\cref{eq.technical3}, respectively. Finally, to prove~\cref{eq.technical4}, we have
    \begin{align*}
        D_eB_e^* + C_eP_e &= \hat{\sigma} U \left[
        \begin{array}{c|c|c}
            \tilde B_1^* & \tilde B_2^* & \hat{B}^*
        \end{array}
        \right]
        +
        \left[
        \begin{array}{c|c|c}
            \tilde C_1 & \tilde C_2 & -\hat{C}
        \end{array}
        \right]
        \left[
        \begin{array}{c|c|c}
        \tilde{\Sigma}_1 & 0 & I \\
        \hline
        0 & \hat{\sigma}I & 0 \\
        \hline
        I & 0 & \tilde{\Sigma}_1\Phi^{-1}
        \end{array}
        \right] \\
        &=
        \left[
        \begin{array}{c|c|c}
            \hat{\sigma}U\tilde B_1^* + \tilde C_1\tilde{\Sigma}_1-\hat{C} & \hat{\sigma}U\tilde B_2^* + \hat{\sigma}\tilde C_2 & \hat{\sigma}U\hat{B}^* + \tilde C_1 - \hat{C}\tilde{\Sigma}_1\Phi^{-1}
        \end{array}
        \right],
    \end{align*}
    where $\hat{\sigma}U\tilde B_1^* \!+\! \tilde C_1\tilde{\Sigma}_1\!-\!\hat{C} \!=\! 0$ by definition, $ \hat{\sigma}U\tilde B_2^* \!+\! \hat{\sigma}\tilde C_2 \!=\! \hat{\sigma} (U\Delta_1^* \!+\! \Delta_2^*\tilde C_2)$ by~\cref{eq.secondU}, and
    \begin{align*}
        &\hat{\sigma}U\hat{B}^* + \tilde C_1 - \hat{C}\tilde{\Sigma}_1\Phi^{-1} = \hat{\sigma}U(\tilde B_1^*\tilde{\Sigma}_1+\hat{\sigma}U^*\tilde C_1)\Phi^{-1} + \tilde C_1 - (\tilde C_1\tilde{\Sigma}_1 + \hat{\sigma}U\tilde B_1^*)\tilde{\Sigma}_1\Phi^{-1} \\
        &\qquad= \hat{\sigma}^2 \tilde C_1 \Phi^{-1} - \hat{\sigma}^2 \Delta_2 \tilde C_1 \Phi^{-1} + \tilde C_1 - \tilde C_1\tilde{\Sigma}_1^2 \Phi^{-1} = - \hat{\sigma}^2 \Delta_2 \tilde C_1 \Phi^{-1}.
    \end{align*}
    This concludes the proof of the lemma.
\end{proof}

Given a square matrix $X$, we define its inertia, denoted by $\Iner(X)$, to be the $3$-tuple consisting of the number of eigenvalues of $X$ whose real parts are $< 0$, $= 0$, and $> 0$, respectively. 
We now prove~\Cref{thm.rankGhat}, which gives criteria for the rank of the system reduced by the modified HNA algorithm being equal to $k$.

\begin{proof}[Proof of~\Cref{thm.rankGhat}]
By~\cref{eq.technical1}, we have the equation
    \begin{equation*}
        -\tilde{\Sigma}_1^{-1} \Phi \hat{A} - \hat{A}^* \Phi\tilde{\Sigma}_1^{-1} =  \tilde{\Sigma}_1^{-1}\Phi \hat{B}\hat{B}^* \Phi\tilde{\Sigma}_1^{-1} + \hat{\sigma}^2 \tilde{\Sigma}_1^{-1}\tilde C_1^* \Delta_2 \tilde C_1\tilde{\Sigma}_1^{-1},
    \end{equation*}
    where $\tilde{\Sigma}_1^{-1}\Phi \hat{B}\hat{B}^* \Phi\tilde{\Sigma}_1^{-1}$ is positive semidefinite and 
    \[
        \big\|{ \hat{\sigma}^2 \tilde{\Sigma}_1^{-1}\tilde C_1^* \Delta_2 \tilde C_1\tilde{\Sigma}_1^{-1}}\big\| < Q_U \big\|{\tilde C_1}\big\|^2.
    \]
    Hence, setting $\alpha_0 = Q_U \big\|{\tilde C_1}\big\|^2$, we have $\tilde{\Sigma}_1^{-1}\Phi \hat{B}\hat{B}^* \Phi\tilde{\Sigma}_1^{-1} + \hat{\sigma}^2 \tilde{\Sigma}_1^{-1}\tilde C_1^* \Delta_2 \tilde C_1\tilde{\Sigma}_1^{-1} + \alpha_0 I$ is positive definite. Since
    \begin{equation*}
        \big(\!-\!\tilde{\Sigma}_1^{-1} \Phi \hat{A} + \alpha_0 I/2\big) + \big(\!-\!\hat{A}^* \Phi\tilde{\Sigma}_1^{-1} \!+ \alpha_0 I/2\big) \!=\! \tilde{\Sigma}_1^{-1}\Phi \hat{B}\hat{B}^* \Phi\tilde{\Sigma}_1^{-1} \!+ \hat{\sigma}^2 \tilde{\Sigma}_1^{-1}\tilde C_1^* \Delta_2 \tilde C_1\tilde{\Sigma}_1^{-1} \!+ \alpha_0 I,
    \end{equation*}
    by~\cite[Cor.~3]{ostrowski1962some}, 
    we have
    \begin{equation}
        \Iner\big(\hat{A}^* - \alpha_0\tilde{\Sigma}_1\Phi^{-1}/2\big) = \Iner\big(-\tilde{\Sigma}_1\Phi^{-1}\big).
    \end{equation}
    By the way we construct $\Phi$, we have $\Iner\big(-\tilde{\Sigma}_1\Phi^{-1}\big) = \Iner\left(-\Phi\right) = (k,0,K-r-k)$. Note that
    \[
        \big\|{\alpha_0\tilde{\Sigma}_1\Phi^{-1}/2}\big\| \leq \big(Q_U \big\|{\tilde C_1}\big\|^2/2\big) \big\|\tilde{\Sigma}_1\Phi^{-1}\big\| \leq Q_U \big\|{\tilde C_1}\big\|^2\delta^{-1}/2.
    \]
    Now, assume $\hat{A}^*$ and $\hat{A}^* - \alpha_0\tilde{\Sigma}_1\Phi^{-1}/2$ have different inertia, then as $\alpha$ ranges from $0$ to $\alpha_0$, by continuity, there exists an $\alpha' \in (0,\alpha_0]$ such that $\hat{A}^* - \alpha'\tilde{\Sigma}_1\Phi^{-1}/2$ contains an eigenvalue with zero real part, but this eigenvalue is then contained in
    \[
        \sigma_{\norm{\alpha'\tilde{\Sigma}_1\Phi^{-1}/2}}(\hat{A}^*) \subset \sigma_{Q_U \norm{\tilde C_1}^2\delta^{-1}/2}(\hat{A}^*),
    \]
    which is a contradiction to~\ref{item.a}. Hence, if~\ref{item.a} holds, then we have
    \[
        \Iner(\hat{A}) = \Iner(\hat{A}^* - \alpha_0\tilde{\Sigma}_1\Phi^{-1}/2) = \Iner\big(-\tilde{\Sigma}_1\Phi^{-1}\big) = (k,0,K-r-k).
    \]

Now, assume~\ref{item.a} does not hold. Then, there exists a $\lambda \in \overline{\sigma_{\rho_1}(\hat{A})} = \sigma_{\rho_1}(\hat{A}^*)$ with $\text{Re}(\lambda) = 0$. This means there exists a unit vector $v$ such that 
    \begin{equation}\label{eq.pseudoexpand1}
        \big\|{(\hat{A}^* - \lambda I)v}\big\| < \rho_1.
    \end{equation}
    Set $\xi = \hat{A}^*v - \lambda v$. Since
    \[
        \hat{A} \tilde{\Sigma}_1 \Phi^{-1} + \tilde{\Sigma}_1 \Phi^{-1} \hat{A}^* + \hat{B}\hat{B}^* = -\hat{\sigma}^2 \Phi^{-1} \tilde C_1^* \Delta_2 \tilde C_1\Phi^{-1},
    \]
    we have
    \begin{align*}
        &-\hat{\sigma}^2 v^* \Phi^{-1} \tilde C_1^* \Delta_2 \tilde C_1 \Phi^{-1} v = v^* \hat{A} \tilde{\Sigma}_1 \Phi^{-1} v + v^* \tilde{\Sigma}_1 \Phi^{-1} \hat{A}^* v + v^* \hat{B}\hat{B}^* v \\
        &= \!(-\lambda v^* \!\!+\! \xi^*) \tilde{\Sigma}_1 \Phi^{-1} \!v \!+\! v^* \tilde{\Sigma}_1 \Phi^{-1} (\lambda v \!+\! \xi) \!+\! v^*\! \hat{B}\hat{B}^* \!v = \xi^*\tilde{\Sigma}_1\Phi^{-1}v \!+\! v^* \tilde{\Sigma}_1\Phi^{-1} \xi \!+\! v^* \!\hat{B}\hat{B}^* \!v.
    \end{align*}
    This shows
    \begin{equation}\label{eq.Cvnorm}
        \begin{aligned}
            \big\|{\hat{B}^* v}\big\|^2 \!\!\leq\! \big\|{\hat{\sigma}^2 v^* \Phi^{-1} \tilde C_1^* \Delta_2 \tilde C_1 \Phi^{-1} v}\big\| \!+\! 2\big\|{\xi^*\tilde{\Sigma}_1\Phi^{-1}v}\big\| \!<\! \big\|{\tilde C_1}\big\|^2\! Q_U \delta^{-2} \!+\! 2\rho_1 \delta^{-1}.
        \end{aligned}
    \end{equation}
    Since $\tilde A_{11} + \hat{A}^* + \tilde B_1\hat{B}^* = 0$, we have
    \begin{align*}
        0 = \tilde A_{11}v + \hat{A}^*v + \tilde B_1\hat{B}^*v = \tilde A_{11}v + \xi + \lambda v + \tilde B_1\hat{B}^*v.
    \end{align*}
    By~\cref{eq.Cvnorm}, we now have
    \begin{equation}
        \begin{aligned}
            \big\|{(\tilde A_{11} + \lambda I) v}\big\| < \rho_1 + \big\|{\tilde B_1}\big\|\sqrt{\big\|{\tilde C_1}\big\|^2 Q_U \delta^{-2} + 2\rho_1 \delta^{-1}} = \rho_2.
        \end{aligned}
    \end{equation}
    Therefore, $-\lambda \in {\sigma_{\rho_2}(\tilde A_{11})}$, which proves the contrapositive of~\ref{item.b}~$\Rightarrow$~\ref{item.a}.

Finally, with the same assumptions of~\cref{eq.pseudoexpand1}, since
    \[
        \tilde A_{21} + \tilde B_2\hat{B}^* = (\hat{\sigma}(\tilde{\Sigma}_2 - \hat{\sigma}I)\tilde A_{21} + (\hat{\sigma}I - \tilde{\Sigma}_2)\tilde A_{12}^* \tilde{\Sigma}_1 + \hat{\sigma}\Delta_1U^*\tilde C_1 + \hat{\sigma} \tilde C_2^* \Delta_2 \tilde C_1)\Phi^{-1},
    \]
    we have
    \begin{equation*}
        \tilde A_{21} v = \underbrace{\big[(\hat{\sigma}(\tilde{\Sigma}_2 \!-\! \hat{\sigma}I)\tilde A_{21} \!+\! (\hat{\sigma}I \!-\! \tilde{\Sigma}_2)\tilde A_{12}^* \tilde{\Sigma}_1 \!+\! \hat{\sigma}\Delta_1U^*\tilde C_1 \!+\! \hat{\sigma} \tilde C_2^* \Delta_2 \tilde C_1)\Phi^{-1} \!-\! \tilde B_2\hat{B}^*\big]}_{\Xi} v,
    \end{equation*}
    which gives us
    \begin{equation*}
    \tilde A 
    \left[
    \begin{array}{c}
         v  \\
         \hline
         0
    \end{array}
    \right]
    =
    \left[
    \begin{array}{c}
         -\lambda v - \xi - \tilde B_1\hat{B}^* v \\
         \hline
         \Xi v
    \end{array}
    \right]
    =
    -\lambda
    \underbrace{
    \left[
    \begin{array}{c}
        v  \\
         \hline
         0
    \end{array}
    \right]}_{v'}
    +
    \underbrace{
    \left[
    \begin{array}{c}
         -\xi - \tilde B_1\hat{B}^* v \\
         \hline
         \Xi v
    \end{array}
    \right]}_{\xi'},
    \end{equation*}
    which implies $(\tilde A + \lambda I)v' = \xi'$.
    Hence, $-\lambda \in {\sigma_{\rho_3}(\tilde A)}$, where
    \begin{equation*}
        \begin{aligned}
            \rho_3 \!=\! \norm{\xi'} &\!\leq\! 2\rho_2 + \delta^{-1} \big(\epsilon\big\|{\tilde A_{21}}\big\| + \epsilon \big\|{\tilde A_{12}}\big\| + Q_E\sqrt{1\!+\!Q_U}\big\|{\tilde C_1}\big\| + \big\|{\tilde C_1}\big\|\big\|{\tilde C_2}\big\|Q_U \big).
        \end{aligned}
    \end{equation*}
    This shows the contrapositive of~\ref{item.c}~$\Rightarrow$~\ref{item.a}.
\end{proof}

\section{Proof of~\Cref{thm.error}}\label{sec:proof2}

In this section, we prove~\Cref{thm.error} on the Hankel approximation error of the modified HNA algorithm. We will continue to use the notations introduced at the beginning of~\cref{sec:proof1}.

\begin{proof}[Proof of~\Cref{thm.error}]
Set $\hat{\sigma} = \tilde{\sigma}_{k+1}$ and define $E(z) = D_e + C_e (zI - A_e)^{-1} B_e$. Then, we have
    \begin{equation}\label{eq.auxiliary}
        E(z) = (\tilde D-\hat{D}) + \tilde C(zI - \tilde A)^{-1}\tilde B - \hat{C}(zI - \hat{A})^{-1}\hat{B} = \tilde{G}(z) - \hat{G}(z).
    \end{equation}
    Hence, for any fixed $z$, we have
    \begin{equation}\label{eq.singval}
        \big\|{\hat{G}(z) \!-\! \tilde{G}(z)}\big\| = \big\|{(\tilde{G}(z) \!-\! \hat{G}(z))(\tilde{G}(z) \!-\! \hat{G}(z))^*}\big\|^{1/2} = \norm{E(z)E(z)^*}^{1/2}.
    \end{equation}
    Let $z \in \C$, $\text{Re}(z) = 0$ be given. We have
    \begin{equation}\label{eq.EzEminusz}
    \begin{aligned}
        &E(z)E(z)^* = (D_e + C_e(zI - A_e)^{-1} B_e)(D^*_e + B^*_e(-zI - A_e^*)^{-1} C^*_e).
    \end{aligned}
    \end{equation}
    Note from the definition of $D_e$ that
    \begin{equation}\label{eq.firstterm}
        D_eD_e^* = (\tilde D-\hat{D})(\tilde D-\hat{D})^* = \hat{\sigma}^2 UU^* = \hat{\sigma}^2 I - \hat{\sigma}^2 \Delta_2.
    \end{equation}
    By~\cref{eq.technical4}, we have
    \begin{equation}\label{eq.secondterm}
        \begin{aligned}
        D_e B^*_e(-zI - A_e^*)^{-1} C^*_e &= (\tilde{\Delta} - C_eP_e) (-zI - A_e^*)^{-1} C^*_e,
    \end{aligned}
    \end{equation}
    Next, by~\cref{eq.perturbedlyapunov}, we have
    \begin{equation}\label{eq.thirdterm}
        \begin{aligned}
        &C_e(zI - A_e)^{-1} B_eB_e^* (-zI - A_e^*)^{-1} C_e^* \\
        & \!=\! C_e(zI \!-\! A_e)^{-1} (\Omega \!+\! (zI - A_e)P_e \!+\! P_e(-\!zI \!-\! A_e^*)) (-\!zI \!-\! A_e^*)^{-1} C_e^* \\
        & \!=\! C_eP_e(-\!zI\!-\!A_e^*)^{-\!1}C_e^* \!+\! C_e(zI \!-\! A_e)^{-\!1}\!P_eC_e^* \!+\! C_e(zI \!-\! A_e)^{-\!1} \Omega (-\!zI\!-\!A_e^*)^{-\!1}C_e^*.
        \end{aligned}
    \end{equation}
    Combining~\cref{eq.EzEminusz},~(\ref{eq.firstterm}),~(\ref{eq.secondterm}), and~(\ref{eq.thirdterm}), we have
    \begin{equation}\label{eq.Efinal}
        \begin{aligned}
            E(z)E(z)^* &= \hat{\sigma}^2 I - \hat{\sigma}^2 \Delta_2 + \tilde{\Delta}(-zI-A_e^*)^{-1}C_e^* + C_e(zI-A_e)^{-1} \tilde{\Delta}^* \\
            &\qquad+ C_e(zI - A_e)^{-1} \Omega (-zI-A_e^*)^{-1}C_e^*.
        \end{aligned}
    \end{equation}
    We prove the theorem by controlling each term added to $\hat{\sigma}^2 I$. By definition, we have 
    \begin{equation}\label{eq.e1}
        e_1 := \norm{\hat{\sigma}^2 \Delta_2} \leq \hat{\sigma}^2 Q_U.
    \end{equation}
    To control the third and the fourth term, we note that
    \begin{align*}
        &\tilde{\Delta}(-zI\!-\!A_e^*)^{-1}\!C_e^* \!=\!\! \left[
        \begin{array}{c|c|c}
            \!\!\!0\!\! & \!\!\hat{\sigma} (U\Delta_1^* \!+\! \Delta_2^*\tilde C_2)\!\!\! & \!\!-\hat{\sigma}^2 \Delta_2 \tilde C_1 \Phi^{-1}\!\!\!\!
        \end{array}
        \right]\!\!
        \left[
        \begin{array}{c|c}
        \!\!\!(zI\!-\!\tilde A)^{-1} \!\!\!\!&\! 0\! \\
        \hline
        \!0 \!&\! \!(zI \!-\! \hat{A})^{-1}\!\!\!\!
        \end{array}
        \right]\!\!\!
        \left[
        \begin{array}{c}
            \!\!\!\tilde C_1^*\!\!\! \\
            \hline
            \!\!\!\tilde C_2^*\!\!\! \\
            \hline
            \!\!\!-\hat{C}^*\!\!\!
        \end{array}
        \right] \\
        &\qquad=\! \left[
        \begin{array}{c|c}
            \!0 & \hat{\sigma} (U\Delta_1^* \!+\! \Delta_2^*\tilde C_2)
        \end{array}
        \right]
        (zI-\tilde A)^{-1} \tilde C^* + \hat{\sigma}^2 \Delta_2 \tilde C_1 \Phi^{-1} (zI \!-\! \hat{A})^{-1} \hat{C}^*,
    \end{align*}
    where the norm of twice the first term can be bounded as
    \[
        e_2 := 2\, \big\|{\!\left[
        \begin{array}{c|c}
            \!\!0 & \hat{\sigma} (U\Delta_1^* \!+\! \Delta_2^*\tilde C_2)\!\!\!
        \end{array}
        \right]
        (zI-\tilde A)^{-1} \tilde C^*}\big\| \leq 2\hat{\sigma} \big(\sqrt{1\!+\!Q_U}Q_E \!+\! Q_U\big\|{\tilde C_2}\big\|\big)\rho^{-1}\big\|{\tilde C}\big\|
    \]
    and the norm of twice the second term can be bounded as
    \[
        e_3 := 2\, \big\|{\hat{\sigma}^2 \Delta_2 \tilde C_1 \Phi^{-1} (zI \!-\! \hat{A})^{-1} \hat{C}^*}\big\| \leq 2\hat{\sigma}^2 Q_U \big\|{\tilde C_1}\big\| \norm{X(z)^{-1}} \big\| \hat{C}\big\|,
    \]
    where $\big\| \hat{C}\big\| \leq \tilde \sigma_1 \big\|{\tilde C_1}\big\| \!+\! \hat{\sigma}\sqrt{1\!+\!Q_U}\big\|{\tilde B_1}\big\|$. Therefore, we have
    \begin{equation}\label{eq.e23}
        \begin{aligned}
            &\big\|{\tilde{\Delta}(-zI-A_e^*)^{-1}C_e^* + C_e(zI-A_e)^{-1} \tilde{\Delta}^*}\big\| \leq e_2 + e_3.
        \end{aligned}
    \end{equation}
    To control the last term in~\cref{eq.Efinal}, we note that
    \begin{align*}
        &C_e(zI - A_e)^{-1} \Omega (-zI-A_e^*)^{-1}C_e^* \\
        &= \tilde C(zI-\tilde A)^{-1}\Omega_1(-zI-\tilde A^*)^{-1}\tilde C^* - \tilde C(zI-\tilde A)^{-1}\Omega_2(-zI-\hat{A}^*)^{-1}\hat{C}^* \\
        &\qquad - \hat{C}(zI-\hat{A})^{-1}\Omega^*_2(-zI-\tilde{A}^*)^{-1}\tilde{C}^* + \hat{C}(zI-\hat{A})^{-1}\Omega_3(-zI-\hat{A}^*)^{-1}\hat{C}^*.
    \end{align*}
    We control each of these terms separately. By~\cref{eq.lyapunov} and~\cref{eq.perturbedlyapunov}, we have
    \[
        \Omega_1 = \left[\!
            \begin{array}{c|c}
                0 &\!\! \hat{\sigma}\tilde A_{12} \!+\! \tilde{\Sigma}_1\tilde A_{21}^* \!-\! \tilde A_{12}\tilde{\Sigma}_2 \!-\! \tilde{\Sigma}_1\tilde A_{21}^* \\
                \hline
                \hat{\sigma}\tilde A_{12}^* \!+\! \tilde A_{21}\tilde{\Sigma}_1 \!-\! \tilde{\Sigma}_2\tilde A_{12}^* \!-\! \tilde A_{21}\tilde{\Sigma}_1 &\!\! \hat{\sigma}\tilde A_{22}^* \!+\! \hat{\sigma}\tilde A_{22} \!-\! \tilde A_{22}\tilde{\Sigma}_2 \!-\! \tilde{\Sigma}_2\tilde A_{22}^*
            \end{array}
            \!\right],
    \]
    whose norm is $\leq 2\epsilon\big(\big\|{\tilde A_{12}}\big\| + \big\|{\tilde A_{21}}\big\| + \big\|{\tilde A_{22}}\big\| \big)$. Hence,
    \begin{equation}\label{eq.e4}
        e_4 := \big\|{\tilde C(zI\!-\!\tilde A)^{-1}\Omega_1(-zI\!-\!\tilde A^*)^{-1}\tilde C^*}\big\| \!\leq\! 2\epsilon\rho^{-2}\big\|{\tilde C}\big\|^2 \big(\big\|{\tilde A_{12}}\big\| + \big\|{\tilde A_{21}}\big\| + \big\|{\tilde A_{22}}\big\|\big).
    \end{equation}
    Next, by~\cref{eq.technical-1} and~(\ref{eq.technical0}), we have
    \[
        \Omega_2 = \left[
        \begin{array}{c}
            0 \\
            \hline
            \hat{\sigma}(\tilde{\Sigma}_2 \!-\! \hat{\sigma}I)\tilde A_{21} \!+\! (\hat{\sigma}I \!-\! \tilde{\Sigma}_2)\tilde A_{12}^* \tilde{\Sigma}_1 \!+\! \hat{\sigma}\Delta_1U^*\tilde C_1 \!+\! \hat{\sigma} \tilde C_2^* \Delta_2 \tilde C_1
        \end{array}
        \right]\Phi^{-1}.
    \]
    Hence,
    \begin{equation}\label{eq.e5}
        \begin{aligned}
        &e_5 := 2\big\|{\tilde C(zI-\tilde A)^{-1}\Omega_2(-zI-\tilde A^*)^{-1} \hat C^*}\big\| \\
        &\leq \!\big\|{\tilde C}\big\| \rho^{-1} \big\|{\hat{\sigma}(\tilde{\Sigma}_2 \!-\! \hat{\sigma}I)\tilde A_{21} \!+\! (\hat{\sigma}I \!-\! \tilde{\Sigma}_2)\tilde A_{12}^* \tilde{\Sigma}_1 \!+\! \hat{\sigma}\Delta_1U^*\tilde C_1 \!+\! \hat{\sigma} \tilde C_2^* \Delta_2 \tilde C_1}\big\| \norm{X(z)^{-1}} \big\|{\hat{C}}\big\| \\
        &\leq \!\big\|{\tilde C}\big\| \rho^{-1} \!\big(\epsilon\hat{\sigma}\big\|{\tilde A_{21}}\big\| \!+\! \epsilon\tilde \sigma_1 \big\|{\tilde A_{12}}\big\| \!+\! \hat{\sigma}Q_E\sqrt{1\!+\!Q_U} \big\|{\tilde C_1}\big\| \!+\! \hat{\sigma}Q_U \!\big\|{\tilde C_1}\big\| \big\|{\tilde C_2}\big\|\big) \!\!\norm{X(z)^{-1}} \!\big\|{\hat{C}}\big\|.
        \end{aligned}
    \end{equation}
    Finally, by~\cref{eq.technical1}, we have $\Omega_3 = -\hat{\sigma}^2 \Phi^{-1} \tilde C_1^* \Delta_2 \tilde C_1\Phi^{-1}$. Hence, we have
    \begin{equation}\label{eq.e6}
        \begin{aligned}
            e_6 := \big\|{\hat{C}(zI-\hat{A})^{-1}\Omega_3(-zI-\hat{A}^*)^{-1}\hat{C}^*}\big\| \leq \hat{\sigma}^2 \big\|\hat{C}\big\|^2 \norm{X(z)^{-1}}^2 \big\|{\tilde C_1}\big\|^2Q_U.
        \end{aligned}
    \end{equation}
Combining~\cref{eq.singval},~\cref{eq.Efinal},~(\ref{eq.e1}),~(\ref{eq.e23}),~(\ref{eq.e4}),~(\ref{eq.e5}) and~(\ref{eq.e6}) and using the triangle inequality, we have $\big\|{\hat{G}(z) \!-\! \tilde{G}(z)}\big\|^2 \leq \hat{\sigma}^2 + \sum_{j=1}^6 e_j$, where for for some $C > 0$ that is bounded by a fixed product of $\big\|{\tilde{A}}\big\|, \big\|{\tilde{B}}\big\|, \big\|{\tilde{C}}\big\|$, $\tilde{\sigma}_1$, $1+Q_E$, and $1+Q_U$, we have
\[
	e_1 \!+\! e_2 \!+\! e_3 \!+\! e_5 \leq C\rho^{-1} \norm{X(z)}^{-1} (\epsilon \!+\! Q_E \!+\! Q_U), \quad e_6 \leq C\norm{X(z)}^{-2}Q_U, \quad e_4 \leq C\rho^{-2}\epsilon.
\]
This completes the proof.
\end{proof}

\section{Proof of~\Cref{thm.QEQU}}\label{sec:proof3}

In this section, we prove the upper bounds on $Q_E$ and $Q_U$ given our construction of $U$ in~\cref{sec:stabilizetheta}. Recall that~\Cref{thm.QEQU} suggests that $Q_E$ and $Q_U$ are on the order of $\epsilon$ under mild assumptions.

\begin{proof}[Proof of~\Cref{thm.QEQU}]
Let $\nu = s_{q} - s_{q+1}$. Let $\tilde B_2 = U_B S_B W_B^*$ be an SVD of $\tilde B_2$ and define $U^{(1)}_B \in \C^{r \times q}, U^{(2)}_B \in \C^{r \times (r-q)}, S^{(1)}_B \in \C^{q \times q}, S^{(2)}_B \in \C^{(r-q) \times (m-q)}, W^{(1)}_B \in \C^{m \times q}, W^{(2)}_B \in \C^{m \times (m-q)}$ so that
\[
     U_B =
    \left[
    \begin{array}{c|c}
         U_B^{(1)} & U_B^{(2)}
    \end{array}
    \right], \quad S_B = 
    \left[
    \begin{array}{c|c}
         S_B^{(1)} & 0 \\
         \hline
         0 & S_B^{(2)}
    \end{array}
    \right], \quad W_B =  \left[
    \begin{array}{c|c}
         W_B^{(1)} & W_B^{(2)}
    \end{array}
    \right].
\]
We have that $U_B S_B^2 U_B^*$ is an SVD of $\tilde B_2 \tilde B_2^*$ and $U_C S_C^2 U_C^*$ is an SVD of $\tilde C_2^* \tilde C_2$. Since by Lyapunov equations~\cref{eq.lyapunov}, we have
    \begin{align*}
        &-\tilde B_2\tilde B_2^* + \tilde C_2^*\tilde C_2 = \tilde A_{22} \tilde{\Sigma}_2 + \tilde{\Sigma}_2 \tilde A_{22}^* - \tilde A_{22}^* \tilde{\Sigma}_2 - \tilde{\Sigma}_2 \tilde A_{22} \\
        &\quad= \tilde A_{22} (\tilde{\Sigma}_2 \!-\! \hat{\sigma}I \!+\! \hat{\sigma} I) + (\tilde{\Sigma}_2 \!-\! \hat{\sigma}I \!+\! \hat{\sigma} I) \tilde A_{22}^* - \tilde A_{22}^* (\tilde{\Sigma}_2 \!-\! \hat{\sigma}I \!+\! \hat{\sigma} I) - (\tilde{\Sigma}_2 \!-\! \hat{\sigma}I \!+\! \hat{\sigma} I) \tilde A_{22} \\
        &\quad = \tilde A_{22} (\tilde{\Sigma}_2 - \hat{\sigma}I) + (\tilde{\Sigma}_2 - \hat{\sigma}I) \tilde A_{22}^* - \tilde A_{22}^* (\tilde{\Sigma}_2 - \hat{\sigma}I) - (\tilde{\Sigma}_2 - \hat{\sigma}I) \tilde A_{22},
    \end{align*}
    we know that
    \[
        \big\|{\tilde B_2 \tilde B_2^* - \tilde C_2^* \tilde C_2}\big\|_F \leq 4\epsilon \big\|{\tilde A_{22}}\big\|_F
    \]
    and $\big\|{\tilde B_2 \tilde B_2^* - \tilde C_2^* \tilde C_2}\big\|_F = 0$ when $\tilde A_{22}$ is Hermitian. Then, by Weyl's inequality and the Davis--Kahan Theorem~\cite{davis1970rotation}, we have 
    \[
        \big\|{U^{(1)}_B(U^{(1)}_B)^* - U^{(1)}_C (U^{(1)}_C)^*}\big\|_F \leq \frac{4\sqrt{2}\epsilon\big\|{\tilde A_{22}}\big\|_F}{\nu - 4\epsilon\big\|{\tilde A_{22}}\big\|_F}
    \]
    and $\big\|{U^{(1)}_B(U^{(1)}_B)^* - U^{(1)}_C (U^{(1)}_C)^*}\big\|_F = 0$ when $\tilde A_{22}$ is Hermitian. To prove~\ref{item.lem1a}, we have
    \begin{equation}\label{eq.BCdiffexpress}
    \begin{aligned}
        \tilde C_2^* U = U_CS_CV_C^*V_CV_B^* = U_CS_CV_B^* = U_C^{(1)} S_C^{(1)}(V_B^{(1)})^* + U_C^{(2)} S_C^{(2)} (V_B^{(2)})^*
    \end{aligned}
    \end{equation}
    By the definition of $V_B^{(1)}$ (see~\cref{eq.constructVB}), we have
    \begin{equation*}
    	\begin{aligned}
		U_C^{(1)} S_C^{(1)}(V_B^{(1)})^* \!\!=\! -U_C^{(1)} (U_C^{(1)})^* \tilde B_2 \!=\! -U_B^{(1)} (U_B^{(1)})^* \! \tilde B_2 \!+\! (U_B^{(1)} (U_B^{(1)})^* \!\!-\! U_C^{(1)} (U_C^{(1)})^*) \tilde B_2,
	\end{aligned}
    \end{equation*}
    where
    \begin{equation*}
    	\begin{aligned}
		&-\!U_B^{(1)} (U_B^{(1)})^* \tilde B_2 \!=\! -U_B^{(1)} [(U_B^{(1)})^* U_B^{(1)}] S_B^{(1)} (W_B^{(1)})^* \!-\! U_B^{(1)} \big[(U_B^{(1)})^* U_B^{(2)}\big] S_B^{(2)} (W_B^{(2)})^* \\
		&=\! -U_B^{(1)} S_B^{(1)} (W_B^{(1)})^* \!-\! U_B^{(2)}S_B^{(2)}(W_B^{(2)})^* \!+\! U_B^{(2)}S_B^{(2)}(W_B^{(2)})^* \!=\! -\tilde{B}_2 \!+\! U_B^{(2)}S_B^{(2)}(W_B^{(2)})^*
	\end{aligned}
    \end{equation*}
    since $(U_B^{(1)})^* U_B^{(1)} = I$ and $(U_B^{(1)})^* U_B^{(2)} = 0$. Thus,~\cref{eq.BCdiffexpress} shows that
    \begin{equation}\label{eq.BCdiff0}
    	\tilde C_2^* U \!=\! -\tilde B_2 \!+\! U_B^{(2)}S_B^{(2)}(W_B^{(2)})^* \!+\! (U_B^{(1)} (U_B^{(1)})^* \!-\! U_C^{(1)} (U_C^{(1)})^*) \tilde B_2 \!+\! U_C^{(2)} S_C^{(2)} (V_B^{(2)})^*.
    \end{equation}
    When $q = r$, we have $U^{(1)}_C(U^{(1)}_C)^* = I$ and $S_C^{(2)} = 0$ so that $\tilde C_2^*U = -\tilde B_2$; otherwise,
    \begin{equation}\label{eq.BCdiff1}
        \big\|{U_C^{(2)} S_C^{(2)} (V_B^{(2)})^*}\big\| \leq \gamma, \qquad \big\|{U_B^{(2)}S_B^{(2)}(W_B^{(2)})^*}\big\| \leq \sqrt{\gamma^2 + 4\epsilon\big\|{\tilde A_{22}}\big\|_F}.
    \end{equation}
    Moreover, we have
    \begin{equation}\label{eq.BCdiff2}
        \big\|{\big(U_B^{(1)} (U_B^{(1)})^* - U_C^{(1)} (U_C^{(1)})^*\big) \tilde B_2}\big\| \leq \frac{4\sqrt{2}\epsilon\big\|{\tilde A_{22}}\big\|_F}{\nu - 4\epsilon\big\|{\tilde A_{22}}\big\|_F} \big\|{\tilde B_2}\big\|
    \end{equation}
    and is equal to zero when $A_{22}$ is Hermitian. Combining~\cref{eq.BCdiff0},~\cref{eq.BCdiff1} and~\cref{eq.BCdiff2}, we proved part~\ref{item.lem1a}.

    Next, we prove part~\ref{item.lem1b}. Since $V_C$ is unitary, we have
    \[
        \norm{I - UU^*} = \norm{V_C^*(I - UU^*)V_C} = \norm{I - V_B^*V_B}.
    \]
    By our construction, we can write $V_B^*V_B$ as
    \begin{equation*}
        V_B^*V_B = 
        \left[
        \begin{array}{c}
            (V_B^{(1)})^* \\
            \hline
            (V_B^{(2)})^*
        \end{array}
        \right]
        \left[
        \begin{array}{c|c}
            V_B^{(1)} & V_B^{(2)}
        \end{array}
        \right]
        =
        \left[
        \begin{array}{c|c}
            (V_B^{(1)})^* V_B^{(1)} & 0 \\
            \hline
            0 & I
        \end{array}
        \right],
    \end{equation*}
    which shows 
    \[
        \norm{I - V_B^*V_B} = \big\|{I - (V_B^{(1)})^* V_B^{(1)}}\big\|.
    \]
    Now, we have
    \begin{align*}
        &(V_B^{(1)})^*V_B^{(1)} = (S_C^{(1)})^{-1} (U_C^{(1)})^* \tilde B_2 \tilde B_2^* U_C^{(1)} (S_C^{(1)})^{-1} \\
        &\quad= (S_C^{(1)})^{-1} (U_C^{(1)})^* \tilde C_2^* \tilde C_2 U_C^{(1)} (S_C^{(1)})^{-1} \!+\! (S_C^{(1)})^{-1} (U_C^{(1)})^* (\tilde B_2 \tilde B_2^* \!-\! \tilde C_2^* \tilde C_2) U_C^{(1)} (S_C^{(1)})^{-1} \\
        &\quad= (S_C^{(1)})^{-1} (U_C^{(1)})^* U_CS_CV_C^* V_CS_C^*U_C^* U_C^{(1)} (S_C^{(1)})^{-1} \\
        &\qquad+ (S_C^{(1)})^{-1} (U_C^{(1)})^* (\tilde B_2 \tilde B_2^* - \tilde C_2^* \tilde C_2) U_C^{(1)} (S_C^{(1)})^{-1}.
    \end{align*}
    The first term in the last expression can be expanded as
    \begin{align*}
        &(S_C^{(1)})^{-1} (U_C^{(1)})^* U_CS_CV_C^* V_CS_C^*U_C^* U_C^{(1)} (S_C^{(1)})^{-1} \\
        &= (S_C^{(1)})^{-1} (U_C^{(1)})^* U_CS_CS_C^*U_C^* U_C^{(1)} (S_C^{(1)})^{-1} \!=\! (S_C^{(1)})^{-1} \!\left[\begin{array}{c|c}
            \!\! I\! & \!0\!\!
        \end{array}\right]\!
        S_CS_C^*
        \left[\begin{array}{c}
             \!I\! \\
             \hline
             \!0\!
        \end{array}\right]\!\!
        (S_C^{(1)})^{-1} \!=\! I.
    \end{align*}
    The residual term can be bounded by
    \begin{align*}
        \norm{I - UU^*} \leq \big\|{(S_C^{(1)})^{-1} (U_C^{(1)})^* (\tilde B_2 \tilde B_2^* - \tilde C_2^* \tilde C_2) U_C^{(1)} (S_C^{(1)})^{-1}}\big\| \leq 4s_q^{-2} \epsilon \big\|{\tilde A_{22}}\big\|
    \end{align*}
    and is zero when $\tilde A_{22}$ is Hermitian. This finishes the proof.
\end{proof}

\section*{Acknowledgments}
We want to thank Serkan Gugercin for discussions about this manuscript. 

\bibliographystyle{siamplain}
\bibliography{references}
\end{document}